\documentclass[article,onefignum,onetabnum]{siamart220329}
\usepackage{_settings}

\title{Computing singular and near-singular integrals over curved boundary elements: The strongly singular case\thanks{Submitted to the editors DATE.}}

\headers{Strongly singular integrals in curved BEM}{H. Montanelli, F. Collino and H. Haddar}

\ifpdf
\hypersetup{
  pdftitle={Computing singular and near-singular integrals over curved boundary elements: The strongly singular case},
  pdfauthor={H. Montanelli, F. Collino, H. Haddar}
}
\fi

\author{Hadrien Montanelli\thanks{Inria, ENSTA Paris, UMA, Institut Polytechnique de Paris, 91120 Palaiseau, France.}
\and Francis Collino\footnotemark[2]
\and Houssem Haddar\footnotemark[2]}

\begin{document}

\maketitle

\begin{abstract}
We present algorithms for computing strongly singular and near-singular surface integrals over curved triangular patches, based on singularity subtraction, the continuation approach, and transplanted Gauss quadrature. We demonstrate the accuracy and robustness of our method for quadratic basis functions and quadratic triangles by integrating it into a boundary element code and solving several scattering problems in 3D. We also give numerical evidence that the utilization of curved boundary elements enhances computational efficiency compared to conventional planar elements.
\end{abstract}



\section{Introduction}

The Helmholtz equation, which has the form
\begin{align}\label{eq:Helmholtz}
\Delta u + k^2u = 0,
\end{align}
appears when one looks for time-harmonic solutions to the wave equation---if $v(\bs{x},t)=u(\bs{x})e^{-i\omega t}$ is a solution to $v_{tt}=c^2\Delta v$, then $u$ satisfies \cref{eq:Helmholtz} with wavenumber $k=\omega /c$. It is of fundamental importance in science and engineering, with applications as diverse as noise scattering, radar and sonar technology, and seismology. For instance, given an incident acoustic wave $u^i$ that is a solution to \cref{eq:Helmholtz} in $\R^3$, the outgoing scattered field $u^s$ generated by a bounded obstacle $D$ is also a solution to \cref{eq:Helmholtz}, in $\R^3\setminus\overline{D}$, with $u^i+u^s=0$ on the boundary $\partial D$ (for sound-soft scattering).

For obstacles $D$ whose complements are connected, a popular technique for calculating scattered fields is based on integral equations. As an example, one can show that the solution $u^s$ to the sound-soft scattering problem of the previous paragraph may be obtained by solving \cite[eq.~(3.51)]{colton1983},
\begin{align}\label{eq:CBIE}
\int_{\partial D}\left\{\frac{\partial G(\bs{x},\bs{y})}{\partial\bs{n}(\bs{y})} - i\eta G(\bs{x},\bs{y})\right\}\varphi^s(\bs{y})dS(\bs{y}) + \frac{\varphi^s(\bs{x})}{2} = -u^i(\bs{x}), \quad \bs{x}\in\partial D,
\end{align}
for some arbitrary real number $\eta\neq0$ such that $\eta\mrm{Re}(k)\geq0$. (The solution to \cref{eq:CBIE} is unique when $\mrm{Im}\,k\geq0$ \cite[Thm.~3.33]{colton1983}.) Once \cref{eq:CBIE} is solved for $\varphi^s$, the scattered field is given by \cite[eq.~(3.49)]{colton1983}
\begin{align}\label{eq:CBIE-sol}
u^s(\bs{x}) =\int_{\partial D}\left\{\frac{\partial G(\bs{x},\bs{y})}{\partial\bs{n}(\bs{y})} - i\eta G(\bs{x},\bs{y})\right\}\varphi^s(\bs{y})dS(\bs{y}), \quad \bs{x}\in\R^3\setminus{\overline{D}}.
\end{align}
The function $G$ in \cref{eq:CBIE} and \cref{eq:CBIE-sol} is the Green's function of the Helmholtz equation in 3D,
\begin{align}\label{eq:kernel}
G(\bs{x},\bs{y}) = \frac{1}{4\pi}\frac{e^{ik\vert\bs{x}-\bs{y}\vert}}{\vert\bs{x}-\bs{y}\vert}.
\end{align}
One of the challenges one faces when solving integral equations of the form of \cref{eq:CBIE} is the computation of singular integrals---when $\bs{x}$ is close or equal to $\bs{y}$, the Green's function and its derivatives become (numerically) unbounded. There are many specialized methods to compute such integrals, including singularity subtraction \cite{aliabadi1985, guiggiani1990, guiggiani1992, hall1991}, singularity cancellation \cite{duffy1982, johnston2013, qin2011, reid2015, xie2020, zhong2024}, and the continuation approach \cite{lenoir2012, rosen1993, rosen1995, salles2013, vijayakumar1989}. Recent works include \cite{faria2021, klockner2013, perezarancibia2019, zhu2022}. 

We proposed, in a previous paper, algorithms to compute weakly singular integrals using singularity subtraction and the continuation approach \cite{montanelli2022}. The term \textit{weakly singular} refers to integrals with singularities of the same type as the Green's function \cref{eq:kernel}. We extend, here, our method to \textit{strongly singular} integrals, which have singularities of the same nature as the derivatives of \cref{eq:kernel}. Specifically, we consider strongly singular and near-singular integrals of the form
\begin{align}\label{eq:intsing2-double-layer}
I(\bs{x}_0) = \int_{\mathcal{T}}\frac{(\bs{x} - \bs{x}_0)\cdot\bs{n}(\bs{x})}{\vert\bs{x}-\bs{x}_0\vert^3}\varphi(F^{-1}(\bs{x}))dS(\bs{x}),
\end{align}
where $\mathcal{T}\subset\R^3$ is a curved triangle defined by a polynomial transformation $F:\widehat{T}\to\mathcal{T}$ of degree $q\geq1$ from some reference planar triangle $\widehat{T}\subset\R^2$, $\bs{x}_0\in\R^3$ is a point on/or close to $\mathcal{T}$, $\varphi:\widehat{T}\to\R$ is a polynomial function of degree $p\geq0$ (not necessarily equal to $q$), and $\bs{n}(\bs{x})$ is the unit normal vector at $\bs{x}$. As in the companion paper \cite{montanelli2022}, our method is based on the computation of the preimage of the singularity in the reference element's space using Newton's method, singularity subtraction with Taylor-like asymptotic expansions, the continuation approach, and transplanted Gauss quadrature. Integrals of the form of \cref{eq:intsing2-double-layer} appear when evaluating the solution \cref{eq:CBIE-sol} at $\bs{x}_0$---we will also look at integrals over two triangles, which occur when solving \cref{eq:CBIE} with a boundary element method \cite{sauter2011}.

One of the main advantages of our method over the standard continuation approach \cite{rosen1995} is that we perform singularity subtraction and continuation after mapping back to the reference triangle, which is both computationally less expensive and much easier to implement. Our novel algorithms are also particularly well-suited for configurations where the singularity is close to the patch's edges. Finally, the benefit of being able to compute strong singularities is that we can solve the Dirichlet problem using the combined boundary integral equation \cref{eq:CBIE}, which is coercive when $k$ is large~\cite{spence2015}. (Coercivity gives explicit bounds one the number of GMRES iterations to achieve a given accuracy.) Our previous work was limited to weak singularities and the (inferior) single-layer formulation.

We start by reviewing the properties of the solid angle in \cref{sec:solid-angle}, which is closely connected to the strongly singular and near-singular integrals \cref{eq:intsing2-double-layer}. We then present our method for computing such integrals in \cref{sec:algorithms}, and provide several numerical examples in \cref{sec:numerics}.

\section{The solid angle}\label{sec:solid-angle}

When $\varphi$ is the constant function $1$, the integral \cref{eq:intsing2-double-layer} simplifies to
\begin{align}
I(\bs{x}_0) = \int_{\mathcal{T}}\frac{(\bs{x}-\bs{x}_0)\cdot\bs{n}(\bs{x})}{\vert\bs{x}-\bs{x}_0\vert^3}dS(\bs{x}) = -\Omega(\bs{x}_0),
\end{align}
where $\Omega(\bs{x}_0)$ is the solid angle of the curved triangle $\mathcal{T}$ subtended at the point $\bs{x}_0$. The solid angle $\Omega(\bs{x}_0)$ cannot be expressed using a closed-form formula for a general curved triangle of degree $q\geq2$. However, in the case of planar triangles where $q=1$, a closed-form formula does exist, which we will examine next to highlight some of the key features of strongly singular integrals.

\subsection{Planar triangles} 

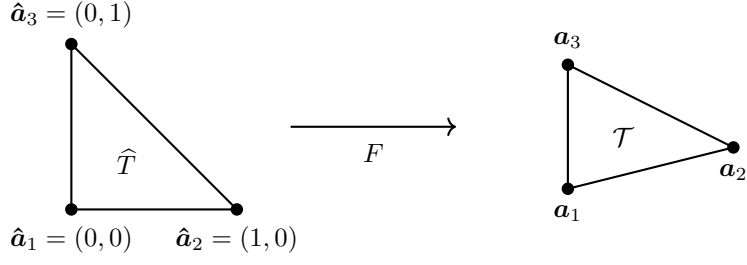
\begin{figure}
\centering
\begin{tikzpicture}[scale = 1.10]
	  \coordinate (a1) at (0, 0); 
	  \coordinate (a2) at (2, 0); 
	  \coordinate (a3) at (0, 2); 
      \draw[thick] (a1) -- (a2) {};
      \draw[thick] (a2) -- (a3) {};
      \draw[thick] (a3) -- (a1) {};
      \node[fill, circle, scale=0.5] at (a1) {};
      \node[anchor=north, yshift=-0.1cm] at (a1) {$\bs{\hat{a}}_1=(0,0)$};
      \node[fill, circle, scale=0.5] at (a2) {};
      \node[anchor=north, yshift=-0.1cm] at (a2) {$\bs{\hat{a}}_2=(1,0)$}; 
      \node[fill, circle, scale=0.5] at (a3) {};
      \node[anchor=south, yshift=0.1cm] at (a3) {$\bs{\hat{a}}_3=(0,1)$};
      \node[anchor=north, yshift=0.2cm] at ($1/3*(a1) + 1/3*(a2) + 1/3*(a3)$) {$\widehat{T}$};
      \draw[thick, ->] (1.65+1, 1) -- (3.65+1, 1) {};
      \node[anchor=north] at (2.65+1, 1-0.1) {$F$}; 
      \def\scl{6}
	  \coordinate (b1) at (0+\scl, 0.25); 
	  \coordinate (b2) at (2+\scl, 0.75); 
	  \coordinate (b3) at (0+\scl, 1.75); 
      \draw[thick] (b1) -- (b2) {};
      \draw[thick] (b2) -- (b3) {};
      \draw[thick] (b3) -- (b1) {};     
      \node[fill, circle, scale=0.5] at (b1) {};
      \node[anchor=north, yshift=-0.1cm] at (b1) {$\bs{a}_1$};
      \node[fill, circle, scale=0.5] at (b2) {};
      \node[anchor=north, yshift=-0.1cm] at (b2) {$\bs{a}_2$};
      \node[fill, circle, scale=0.5] at (b3) {};
      \node[anchor=south, yshift=0.1cm] at (b3) {$\bs{a}_3$};
	  \node[anchor=north, yshift=0.2cm] at ($1/3*(b1) + 1/3*(b2) + 1/3*(b3)$) {$\mathcal{T}$};
\end{tikzpicture}
\caption{\textit{An arbitrary planar triangle $\mathcal{T}$ is obtained from the planar reference triangle $\widehat{T}$ via the linear map $F$ defined in \cref{eq:linear-map}. The $\bs{\hat{a}}_j$'s verify $\lambda_i(\bs{\hat{a}}_j)=\delta_{ij}$, while the $\bs{a}_j$'s verify $\bs{a}_j=F(\bs{\hat{a}}_j)$.}}
\label{fig:planar-tri}
\end{figure}

Let $\widehat{T}$ be the reference planar triangle,
\begin{align}
\widehat{T} = \{(\hat{x}_1,\hat{x}_2) \, : \, 0\leq\hat{x}_1\leq1, \; 0\leq\hat{x}_2\leq1-\hat{x}_1\} \subset\R^2.
\end{align}
An arbitrary planar triangle $\mathcal{T}\subset\R^3$ may be defined by three points $\bs{a}_j\in\R^3$ and a linear transformation $F:\widehat{T}\mapsto\mathcal{T}$,
\begin{align}\label{eq:linear-map}
F(\bs{\hat{x}}) = \sum_{j=1}^3\lambda_j(\bs{\hat{x}})\bs{a}_j\in\R^3,
\end{align}
with $\bs{\hat{x}}=(\hat{x}_1,\hat{x}_2)$, and where the $\lambda_j$'s are the linear basis functions defined on $\widehat{T}$ by
\begin{align}\label{eq:linear-func}
\lambda_1(\bs{\hat{x}}) = 1 - \hat{x}_1 - \hat{x}_2, \quad
\lambda_2(\bs{\hat{x}}) = \hat{x}_1, \quad
\lambda_3(\bs{\hat{x}}) = \hat{x}_2. 
\end{align}
We illustrate this in \cref{fig:planar-tri}. Note that $F(\bs{\hat{x}}) = J\bs{\hat{x}} + \bs{a}_1$ with a constant Jacobian matrix
\begin{align}
J = \left(\begin{array}{c|c}
& \\
\bs{a}_2 - \bs{a}_1 & \bs{a}_3 - \bs{a}_1 \\
&
\end{array}\right)\in\R^{3\times2}.
\end{align}
For such a triangle, the solid angle
\begin{align}
\Omega(\bs{x}_0) = \int_{\mathcal{T}}\frac{(\bs{x}_0-\bs{x})\cdot\bs{n}}{\vert\bs{x}-\bs{x}_0\vert^3}dS(\bs{x}), \quad
\bs{n} = \frac{(\bs{a}_2-\bs{a}_1)\times(\bs{a}_3-\bs{a}_2)}{\vert(\bs{a}_2-\bs{a}_1)\times(\bs{a}_3-\bs{a}_2)\vert},
\end{align}
may be computed via the following formula \cite{vanoosterom1983},
\begin{align}\label{eq:solid-angle}
\Omega(\bs{x}_0) = \mrm{sign}((\bs{x}_0 - \bs{a}_1)\cdot\bs{n})\times(-\pi + \phi_1(\bs{x}_0) + \phi_2(\bs{x}_0) + \phi_3(\bs{x}_0)),
\end{align}
with dihedral angles
\begin{align}
\phi_1 = \mathrm{acos}\left(\frac{c_1 - c_2c_3}{s_2s_3}\right), \quad
\phi_2 = \mathrm{acos}\left(\frac{c_2 - c_1c_3}{s_1s_3}\right), \quad
\phi_3 = \mathrm{acos}\left(\frac{c_3 - c_1c_2}{s_1s_2}\right).
\end{align}
The $c_i$'s above are given by
\begin{align}
c_1 = \frac{(\bs{a}_2 - \bs{x}_0)\cdot(\bs{a}_3 - \bs{x}_0)}{\big\vert\bs{a}_2 - \bs{x}_0\big\vert\big\vert\bs{a}_3 - \bs{x}_0\big\vert}, \quad
c_2 = \frac{(\bs{a}_1 - \bs{x}_0)\cdot(\bs{a}_3 - \bs{x}_0)}{\big\vert\bs{a}_1 - \bs{x}_0\big\vert\big\vert\bs{a}_3 - \bs{x}_0\big\vert}, \quad
c_3 = \frac{(\bs{a}_1 - \bs{x}_0)\cdot(\bs{a}_2 - \bs{x}_0)}{\big\vert\bs{a}_1 - \bs{x}_0\big\vert\big\vert\bs{a}_2 - \bs{x}_0\big\vert},
\end{align}
while the $s_i$'s are given by $s_i=\sqrt{1 - c_i^2}$. 

An important consequence of \cref{eq:solid-angle} is that, as $\bs{x}_0$ approaches the interior of $\mathcal{T}$ from above (with respect to its normal), $\Omega(\bs{x}_0)$ tends to $2\pi$, since each of the dihedral angles tends to $\pi$. On the other hand, when $\bs{x}_0$ approaches an edge or the vertex $\bs{a}_i$ of $\mathcal{T}$, $\Omega(\bs{x}_0)$ tends to $\pi$ or to the dihedral angle $\phi_i$, respectively. Finally, $\Omega(\bs{x}_0)$ goes to $0$ when $\bs{x}_0$ approaches a point in the plane of $\mathcal{T}$ while staying outside of $\mathcal{T}$. (These results have to be changed to their negatives if $\bs{x}_0$ approaches $\mathcal{T}$ from below.) To summarize, we have the following limits:
\begin{align}\label{eq:solid-angle-limits}
& \Omega(\bs{x}_0) \to \pm2\pi, && \text{as $\bs{x}_0$ approaches the interior of $\mathcal{T}$}, \\
& \Omega(\bs{x}_0) \to \pm\pi, && \text{as $\bs{x}_0$ approaches an edge of $\mathcal{T}$ but not a vertex}, \nonumber \\
& \Omega(\bs{x}_0) \to \pm\phi_i, && \text{as $\bs{x}_0$ approaches the vertex $\bs{a}_i$ of $\mathcal{T}$}, \nonumber \\
& \Omega(\bs{x}_0) \to 0, && \text{as $\bs{x}_0$ approaches the exterior $\mathcal{T}$}. \nonumber
\end{align}
Therefore, the solid angle is discontinuous through $\mathcal{T}$ but, in view of \cref{eq:solid-angle-limits}, we may set
\begin{align}\label{eq:solid-angle-zero}
\Omega(\bs{x}_0) = \frac{\lim_{\bs{x}_0\to\mathcal{T}^+}\Omega(\bs{x}_0) + \lim_{\bs{x}_0\to\mathcal{T}^-}\Omega(\bs{x}_0)}{2} = 0, \quad \bs{x}_0\in\mathcal{T}.
\end{align}
(In this formula, $\mathcal{T}^\pm$ means that the point $\bs{x}_0$ approaches the triangle $\mathcal{T}$ from above/below.)

Another way to obtain the limits \cref{eq:solid-angle-limits}, which will be useful for curved triangles, is to map the solid angle back to the reference triangle $\widehat{T}$,
\begin{align}
\Omega(\bs{x}_0) = \int_{\widehat{T}}\frac{(\bs{x}_0 - F(\bs{\hat{x}}))\cdot\bs{\hat{n}}}{\vert F(\bs{\hat{x}})-\bs{x}_0\vert^3}dS(\bs{\hat{x}}), \quad \bs{\hat{n}}=J_1\times J_2,
\end{align}
and write $\bs{x}_0$ as the sum of a vector parallel to $\mathcal{T}$ and a vector orthogonal to $\mathcal{T}$,
\begin{align}
\bs{x}_0 = F(\bs{\hat{x}}_0) + h\frac{\bs{\hat{n}}}{\vert\bs{\hat{n}}\vert},
\end{align} 
for some scalar $h$. Using $F(\bs{\hat{x}})=J\bs{\hat{x}} + \bs{a}_1$, we get $\Omega(\bs{\hat{x}}_0, h) = \Omega_{-1}(\bs{\hat{x}}_0, h) + \Omega_{-2}(\bs{\hat{x}}_0, h)$ with
\begin{align}
\Omega_{-1}(\bs{\hat{x}}_0, h) = \int_{\widehat{T}}\frac{\left[J(\bs{\hat{x}}_0 - \bs{\hat{x}})\right]\cdot\bs{\hat{n}}}{\left[\vert J(\bs{\hat{x}}-\bs{\hat{x}}_0)\vert^2 + h^2\right]^{\frac{3}{2}}}dS(\bs{\hat{x}}), \\
\Omega_{-2}(\bs{\hat{x}}_0, h) = \int_{\widehat{T}}\frac{h\vert\bs{\hat{n}}\vert}{\left[\vert J(\bs{\hat{x}}-\bs{\hat{x}}_0)\vert^2 + h^2\right]^{\frac{3}{2}}}dS(\bs{\hat{x}}). \nonumber 
\end{align}
The term $\Omega_{-1}(\bs{\hat{x}}_0, h)$ vanishes for all $\bs{\hat{x}}_0$ and $h$ since $J(\bs{\hat{x}}_0 - \bs{\hat{x}})$ is orthogonal to the normal $\bs{\hat{n}}$. The term $\Omega_{-2}(\bs{\hat{x}}_0, h)$ is strongly singular and we recover the limits \cref{eq:solid-angle-limits} by noting that the integrand approaches the radially symmetric Dirac delta as $h\to0^\pm$. In other words, we have
\begin{align}
\Omega(\bs{\hat{x}}_0, h) = \int_{\widehat{T}}\frac{h\vert\bs{\hat{n}}\vert}{\left[\vert J(\bs{\hat{x}}-\bs{\hat{x}}_0)\vert^2 + h^2\right]^{\frac{3}{2}}}dS(\bs{\hat{x}}),
\end{align}
with the limits
\begin{align}\label{eq:solid-angle-limits-2}
& \lim_{h\to0^\pm}\Omega(\bs{\hat{x}}_0,h) \to \pm2\pi, && \text{if $\bs{\hat{x}}_0$ is in the interior of $\widehat{T}$}, \\
& \lim_{h\to0^\pm}\Omega(\bs{\hat{x}}_0,h) \to \pm\pi, && \text{if $\bs{\hat{x}}_0$ is on an edge of $\widehat{T}$ but is not a vertex}, \nonumber \\
& \lim_{h\to0^\pm}\Omega(\bs{\hat{x}}_0,h) \to \pm\phi_i, && \text{if $\bs{\hat{x}}_0$ is a vertex of $\widehat{T}$, that is, $F(\bs{\hat{x}}_0)=\bs{a}_i$}, \nonumber \\
& \lim_{h\to0^\pm}\Omega(\bs{\hat{x}}_0,h) \to 0, && \text{if $\bs{\hat{x}}_0$ is in the exterior $\widehat{T}$}. \nonumber
\end{align}
Once again, we observe that $\Omega(\bs{\hat{x}}_0, h)$ is discontinuous at $h=0$ but, in view of \cref{eq:solid-angle-limits-2}, we may set
\begin{align}\label{eq:solid-angle-zero-2}
\Omega(\bs{\hat{x}}_0,0)=0, \quad \bs{\hat{x}}_0\in\widehat{T}.
\end{align}
We summarize the limits \cref{eq:solid-angle-limits-2} and \cref{eq:solid-angle-zero-2} in \cref{tab:solid-angle-planar}. We may be tempted to say that the solid angle is merely weakly singular because of the cancellation when $h=0$. However, in practice, it will act as if it were strongly singular when $h$ is small. We will come back to this in \cref{sec:numerics}.

\begin{table}
\caption{\textit{Limiting values of the solid angle $\Omega(\bs{\hat{x}}_0,h)$ for planar triangles.}}
\label{tab:solid-angle-planar}
\centering
\ra{1.3}
\begin{tabular}{c|c|c}
\toprule
$h\to0^+$ & $h\to0^-$ & $h=0$ \\
\midrule
\begin{tikzpicture}[scale = 0.9]
      \draw[pattern=north west lines, pattern color=lightgray, draw=none] (-1,-1) -- (3,-1) -- (3,0) -- (0,0) -- (0,2) -- (2,0) -- (3,0) -- (3,3) -- (-1,3) -- (-1,-1);
      \draw[thick] (0, 0) -- (2, 0) {};
      \draw[thick] (0, 0) -- (0, 2) {};
      \draw[thick] (2, 0) -- (0, 2) {};
      \node[fill, circle, scale=0.5] at (0, 0) {};
      \node[anchor=north] at (0, 0-0.1) {$\phi_1$};
      \node[fill, circle, scale=0.5] at (2, 0) {};
      \node[anchor=north] at (2, 0-0.1) {$\phi_2$};
      \node[fill, circle, scale=0.5] at (0, 2) {};
      \node[anchor=south] at (0, 2+0.1) {$\phi_3$};
      \node[anchor=north] at (1, 0-0.1) {$\pi$};
      \node[anchor=west] at (1+0.2, 1) {$\pi$};
      \node[anchor=east] at (0-0.1, 1) {$\pi$};
      \node[anchor=north] at (2/3, 3/3) {$2\pi$};
      \node[anchor=north, color=lightgray] at (2, 2.5) {\Large{$0$}};
      \node[anchor=north] at (0,3.15) {$\phantom{1}$};
\end{tikzpicture} 
& \begin{tikzpicture}[scale = 0.9]
      \draw[pattern=north west lines, pattern color=lightgray, draw=none] (-1,-1) -- (3,-1) -- (3,0) -- (0,0) -- (0,2) -- (2,0) -- (3,0) -- (3,3) -- (-1,3) -- (-1,-1);
      \draw[thick] (0, 0) -- (2, 0) {};
      \draw[thick] (0, 0) -- (0, 2) {};
      \draw[thick] (2, 0) -- (0, 2) {};
      \node[fill, circle, scale=0.5] at (0, 0) {};
      \node[anchor=north] at (0, 0-0.1) {$-\phi_1$};
      \node[fill, circle, scale=0.5] at (2, 0) {};
      \node[anchor=north] at (2, 0-0.1) {$-\phi_2$};
      \node[fill, circle, scale=0.5] at (0, 2) {};
      \node[anchor=south] at (0, 2+0.1) {$-\phi_3$};
      \node[anchor=north] at (1, 0-0.1) {$-\pi$};
      \node[anchor=west] at (1+0.2, 1) {$-\pi$};
      \node[anchor=east] at (0-0.1, 1) {$-\pi$};
      \node[anchor=north] at (2/3, 3/3) {$-2\pi$};
      \node[anchor=north, color=lightgray] at (2, 2.5) {\Large{$0$}};
\end{tikzpicture} 
& \begin{tikzpicture}[scale = 0.9]
      \draw[pattern=north west lines, pattern color=lightgray, draw=none] (-1,-1) -- (3,-1) -- (3,3) -- (-1,3);
      \draw[thick] (0, 0) -- (2, 0) {};
      \draw[thick] (0, 0) -- (0, 2) {};
      \draw[thick] (2, 0) -- (0, 2) {};
      \node[fill, circle, scale=0.5] at (0, 0) {};
      \node[fill, circle, scale=0.5] at (2, 0) {};
      \node[fill, circle, scale=0.5] at (0, 2) {};
      \node[anchor=north, color=lightgray] at (2, 2.5) {\Large{$0$}};
\end{tikzpicture} \\
\bottomrule
\end{tabular}
\end{table}

\subsection{Curved triangles} 

Throughout the paper, we will utilize quadratic triangles to illustrate our results for curved triangles. A quadratic triangle $\mathcal{T}\subset\R^3$ may be defined by six points $\bs{a}_j\in\R^3$ and a quadratic map $F:\widehat{T}\to\mathcal{T}$,
\begin{align}\label{eq:quad-map}
F(\bs{\hat{x}}) = \sum_{j=1}^6\varphi_j(\bs{\hat{x}})\bs{a}_j\in\R^3,
\end{align}
where the $\varphi_j$'s are the quadratic basis functions defined on $\widehat{T}$ by
\begin{align}\label{eq:quad-func}
& \varphi_1(\bs{\hat{x}}) = \lambda_1(\bs{\hat{x}})(2\lambda_1(\bs{\hat{x}}) - 1), \quad\quad \varphi_4(\bs{\hat{x}}) = 4\lambda_1(\bs{\hat{x}})\lambda_2(\bs{\hat{x}}), \\
& \varphi_2(\bs{\hat{x}}) = \lambda_2(\bs{\hat{x}})(2\lambda_2(\bs{\hat{x}}) - 1), \quad\quad \varphi_5(\bs{\hat{x}}) = 4\lambda_2(\bs{\hat{x}})\lambda_3(\bs{\hat{x}}), \nonumber \\
& \varphi_3(\bs{\hat{x}}) = \lambda_3(\bs{\hat{x}})(2\lambda_3(\bs{\hat{x}}) - 1), \quad\quad \varphi_6(\bs{\hat{x}}) = 4\lambda_1(\bs{\hat{x}})\lambda_3(\bs{\hat{x}}), \nonumber
\end{align}
and the $\lambda_j$'s are defined in \cref{eq:linear-func}. We illustrate this in \cref{fig:curved-tri}. Here, the Jacobian matrix $J$ reads
\begin{align}
J(\bs{\hat{x}}) = \left(
\begin{array}{c|c}
& \\
\hspace{-0.15cm} F_{\hat{x}_1}(\bs{\hat{x}}) & F_{\hat{x}_2}(\bs{\hat{x}}) \hspace{-0.15cm}\phantom{} \\
&
\end{array}
\right)\in\R^{3\times2},
\end{align}
with (componentwise) partial derivatives $F_{\hat{x}_1}(\bs{\hat{x}})\in\R^3$ and $F_{\hat{x}_2}(\bs{\hat{x}})\in\R^3$ with respect to $\hat{x}_1$ and $\hat{x}_2$.

\begin{figure}
\centering
\begin{tikzpicture}[scale = 1.10]
      \draw[thick] (0, 0) -- (2, 0) {};
      \draw[thick] (0, 0) -- (0, 2) {};
      \draw[thick] (2, 0) -- (0, 2) {};
      \node[fill, circle, scale=0.5] at (0, 0) {};
      \node[anchor=north] at (0, 0-0.1) {$\bs{\hat{a}}_1$};
      \node[fill, circle, scale=0.5] at (2, 0) {};
      \node[anchor=north] at (2, 0-0.1) {$\bs{\hat{a}}_2$};
      \node[fill, circle, scale=0.5] at (0, 2) {};
      \node[anchor=south] at (0, 2+0.1) {$\bs{\hat{a}}_3$};
      \node[fill, circle, scale=0.5] at (1, 0) {};
      \node[anchor=north] at (1, 0-0.1) {$\bs{\hat{a}}_4$};
      \node[fill, circle, scale=0.5] at (1, 1) {};
      \node[anchor=west] at (1+0.2, 1) {$\bs{\hat{a}}_5$};
      \node[fill, circle, scale=0.5] at (0, 1) {};
      \node[anchor=east] at (0-0.1, 1) {$\bs{\hat{a}}_6$};
      \node[anchor=north] at (2/3, 3/3) {$\widehat{T}$};
      \draw[thick, ->] (1.65+1, 1) -- (3.65+1, 1) {};
      \node[anchor=north] at (2.65+1, 1-0.1) {$F$};
      \draw[thick] (4+2, 0) -- (6+2, 0) {};
      \draw[thick] (4+2, 0) -- (4+2, 2) {};
      \draw[thick, domain=-1:1, smooth, variable=\theta, black] plot ({6+2*(1-(\theta+1)/2) + 2*2*(2*0.7-1)*(1-(\theta+1)/2)*(\theta+1)/2}, {2*(\theta+1)/2 + 2*2*(2*0.6-1)*(1-(\theta+1)/2)*(\theta+1)/2});
      \node[fill, circle, scale=0.5] at (4+2, 0) {};
      \node[anchor=north] at (4+2, 0-0.1) {$\bs{a}_1$};
      \node[fill, circle, scale=0.5] at (6+2, 0) {};
      \node[anchor=north] at (6+2, 0-0.1) {$\bs{a}_2$};
      \node[fill, circle, scale=0.5] at (4+2, 2) {};
      \node[anchor=south] at (4+2, 2+0.1) {$\bs{a}_3$};
      \node[fill, circle, scale=0.5] at (5+2, 0) {};
      \node[anchor=north] at (5+2, 0-0.1) {$\bs{a}_4$};
      \node[fill, circle, scale=0.5] at (6+2*0.7, 2*0.6) {};
      \node[anchor=west] at (6+2*0.7+0.2, 2*0.6) {$\bs{a}_5$};
      \node[fill, circle, scale=0.5] at (4+2, 1) {};
      \node[anchor=east] at (4-0.1+2, 1) {$\bs{a}_6$};
      \node[anchor=north] at (2/3+4+.2+2, 3/3) {$\mathcal{T}$};
\end{tikzpicture}
\caption{\textit{A quadratic triangle $\mathcal{T}$ is obtained from $\widehat{T}$ via the quadratic map $F$ \cref{eq:quad-map}. Note that $\bs{\hat{a}}_4$, $\bs{\hat{a}}_5$, and $\bs{\hat{a}}_6$ are the midpoints. Again, the $\bs{\hat{a}}_j$'s verify $\varphi_i(\bs{\hat{a}}_j)=\delta_{ij}$, while the $\bs{a}_j$'s verify $\bs{a}_j=F(\bs{\hat{a}}_j)$.}}
\label{fig:curved-tri}
\end{figure}
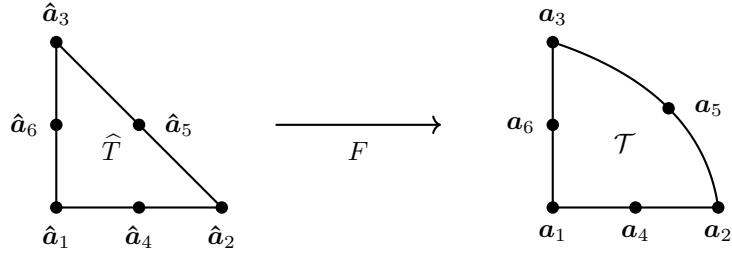

Similarly to what we did for planar triangles, we may write the solid angle as
\begin{align}
\Omega(\bs{x}_0) = \int_{\widehat{T}}\frac{(\bs{x}_0 - F(\bs{\hat{x}}))\cdot\bs{\hat{n}}(\bs{\hat{x}})}{\vert F(\bs{\hat{x}})-\bs{x}_0\vert^3}dS(\bs{\hat{x}}), \quad
\bs{\hat{n}}(\bs{\hat{x}}) = F_{\hat{x}_1}(\bs{\hat{x}})\times F_{\hat{x}_2}(\bs{\hat{x}}),
\end{align}
and decompose $\bs{x}_0$ as the sum of a vector parallel to the tangent plane to the triangle $\mathcal{T}$ at $F(\bs{\hat{x}}_0)$ and a vector orthogonal to that plane,
\begin{align}
\bs{x}_0 = F(\bs{\hat{x}}_0) + h\frac{\bs{\hat{n}}_0}{\vert\bs{\hat{n}}_0\vert}, \quad \bs{\hat{n}}_0 = \bs{\hat{n}}(\bs{\hat{x}}_0).
\end{align}
This yields $\Omega(\bs{\hat{x}}_0, h) = \Omega_{-1}(\bs{\hat{x}}_0, h) + \Omega_{-2}(\bs{\hat{x}}_0, h)$ with
\begin{align}
\Omega_{-1}(\bs{\hat{x}}_0, h) = \int_{\widehat{T}}\frac{(F(\bs{\hat{x}}_0) - F(\bs{\hat{x}}))\cdot\bs{\hat{n}}(\bs{\hat{x}})}{\vert F(\bs{\hat{x}}) - F(\bs{\hat{x}}_0) - h\bs{\hat{n}}_0/\vert\bs{\hat{n}}_0\vert\vert^3}dS(\bs{\hat{x}}), \\
\Omega_{-2}(\bs{\hat{x}}_0, h) = \int_{\widehat{T}}\frac{h\bs{\hat{n}}_0/\vert\bs{\hat{n}}_0\vert\cdot\bs{\hat{n}}(\bs{\hat{x}})}{\vert F(\bs{\hat{x}}) - F(\bs{\hat{x}}_0) - h\bs{\hat{n}}_0/\vert\bs{\hat{n}}_0\vert\vert^3}dS(\bs{\hat{x}}). \nonumber
\end{align}
The term $\Omega_{-1}(\bs{\hat{x}}_0, h)$ is weakly singular, continuous at $h=0$, and does not vanish in general, while the term $\Omega_{-2}(\bs{\hat{x}}_0, h)$ is strongly singular, discontinuous at $h=0$, and has the same limits as \cref{eq:solid-angle-limits-2}. This can be shown by nothing that
\begin{align}
\int_{\widehat{T}}\frac{h\bs{\hat{n}}_0/\vert\bs{\hat{n}}_0\vert\cdot\bs{\hat{n}}(\bs{\hat{x}})}{\vert F(\bs{\hat{x}}) - F(\bs{\hat{x}}_0) - h\bs{\hat{n}}_0/\vert\bs{\hat{n}}_0\vert\vert^3}dS(\bs{\hat{x}}) 
= \int_{\widehat{T}}\frac{h\vert\bs{\hat{n}}_0\vert}{\left[\vert J_0(\bs{\hat{x}} - \bs{\hat{x}}_0)\vert^2+h^2\right]^{\frac{3}{2}}}dS(\bs{\hat{x}}) + \OO(h),
\end{align}
where $J_0=J(\bs{\hat{x}}_0)$. Again, we may set $\Omega_{-2}(\bs{\hat{x}}_0,0)=0$ for all $\bs{\hat{x}}_0\in\widehat{T}$; see \cref{tab:solid-angle-curved}. The solid angle is mathematically weakly singular but numerically strongly singular, as we will see in \cref{sec:numerics}.

\begin{table}
\caption{\textit{Limiting values of the solid angle $\Omega(\bs{\hat{x}}_0,h) = \Omega_{-1}(\bs{\hat{x}}_0,h) + \Omega_{-2}(\bs{\hat{x}}_0,h)$ for curved triangles.}}
\label{tab:solid-angle-curved}
\centering
\ra{1.3}
\begin{tabular}{c|c|c}
\toprule
$h\to0^+$ & $h\to0^-$ & $h=0$ \\
\midrule
$\Omega_{-1}(\bs{\hat{x}}_0,0^+)\neq0$ in general & $\Omega_{-1}(\bs{\hat{x}}_0,0^-)\neq0$ in general & $\Omega_{-1}(\bs{\hat{x}}_0,0)\neq0$ in general \\
values of $\Omega_{-2}(\bs{\hat{x}}_0,0^+)$: & values of $\Omega_{-2}(\bs{\hat{x}}_0,0^-)$: & values of $\Omega_{-2}(\bs{\hat{x}}_0,0)$: \\
\begin{tikzpicture}[scale = 0.9]
      \draw[pattern=north west lines, pattern color=lightgray, draw=none] (-1,-1) -- (3,-1) -- (3,0) -- (0,0) -- (0,2) -- (2,0) -- (3,0) -- (3,3) -- (-1,3) -- (-1,-1);
      \draw[thick] (0, 0) -- (2, 0) {};
      \draw[thick] (0, 0) -- (0, 2) {};
      \draw[thick] (2, 0) -- (0, 2) {};
      \node[fill, circle, scale=0.5] at (0, 0) {};
      \node[anchor=north] at (0, 0-0.1) {$\phi_1$};
      \node[fill, circle, scale=0.5] at (2, 0) {};
      \node[anchor=north] at (2, 0-0.1) {$\phi_2$};
      \node[fill, circle, scale=0.5] at (0, 2) {};
      \node[anchor=south] at (0, 2+0.1) {$\phi_3$};
      \node[anchor=north] at (1, 0-0.1) {$\pi$};
      \node[anchor=west] at (1+0.2, 1) {$\pi$};
      \node[anchor=east] at (0-0.1, 1) {$\pi$};
      \node[anchor=north] at (2/3, 3/3) {$2\pi$};
      \node[anchor=north, color=lightgray] at (2, 2.5) {\Large{$0$}};
      \node[anchor=north] at (0,3.15) {$\phantom{1}$};
\end{tikzpicture} 
& \begin{tikzpicture}[scale = 0.9]
      \draw[pattern=north west lines, pattern color=lightgray, draw=none] (-1,-1) -- (3,-1) -- (3,0) -- (0,0) -- (0,2) -- (2,0) -- (3,0) -- (3,3) -- (-1,3) -- (-1,-1);
      \draw[thick] (0, 0) -- (2, 0) {};
      \draw[thick] (0, 0) -- (0, 2) {};
      \draw[thick] (2, 0) -- (0, 2) {};
      \node[fill, circle, scale=0.5] at (0, 0) {};
      \node[anchor=north] at (0, 0-0.1) {$-\phi_1$};
      \node[fill, circle, scale=0.5] at (2, 0) {};
      \node[anchor=north] at (2, 0-0.1) {$-\phi_2$};
      \node[fill, circle, scale=0.5] at (0, 2) {};
      \node[anchor=south] at (0, 2+0.1) {$-\phi_3$};
      \node[anchor=north] at (1, 0-0.1) {$-\pi$};
      \node[anchor=west] at (1+0.2, 1) {$-\pi$};
      \node[anchor=east] at (0-0.1, 1) {$-\pi$};
      \node[anchor=north] at (2/3, 3/3) {$-2\pi$};
      \node[anchor=north, color=lightgray] at (2, 2.5) {\Large{$0$}};
\end{tikzpicture} 
& \begin{tikzpicture}[scale = 0.9]
      \draw[pattern=north west lines, pattern color=lightgray, draw=none] (-1,-1) -- (3,-1) -- (3,3) -- (-1,3);
      \draw[thick] (0, 0) -- (2, 0) {};
      \draw[thick] (0, 0) -- (0, 2) {};
      \draw[thick] (2, 0) -- (0, 2) {};
      \node[fill, circle, scale=0.5] at (0, 0) {};
      \node[fill, circle, scale=0.5] at (2, 0) {};
      \node[fill, circle, scale=0.5] at (0, 2) {};
      \node[anchor=north, color=lightgray] at (2, 2.5) {\Large{$0$}};
\end{tikzpicture} \\
\bottomrule
\end{tabular}
\end{table}

\section{Computing strongly singular and near-singular integrals}\label{sec:algorithms}

To compute strongly singular and near-singular integrals of the form of \cref{eq:intsing2-double-layer}, we proceed in five steps, as in \cite{montanelli2022} for the computation of weakly singular and near-singular integrals. We quickly review the five steps, highlighting some of the differences; for details, we refer to \cite[sect.~2]{montanelli2022}.

\paragraph{Step 1.~Mapping back} We map $\mathcal{T}$ back to the reference element $\widehat{T}$. The integral \cref{eq:intsing2-double-layer} becomes
\begin{align}\label{eq:step1}
I(\bs{x}_0) = \int_{\widehat{T}}\frac{(F(\bs{\hat{x}})-\bs{x}_0)\cdot\bs{\hat{n}}(\bs{\hat{x}})}{\vert F(\bs{\hat{x}})-\bs{x}_0\vert^3}\varphi(\bs{\hat{x}})dS(\bs{\hat{x}}), \quad
\bs{\hat{n}}(\bs{\hat{x}}) = F_{\hat{x}_1}(\bs{\hat{x}})\times F_{\hat{x}_2}(\bs{\hat{x}}).
\end{align}

\paragraph{Step 2.~Locating the singularity} We compute $\bs{\hat{x}}_0\in\R^2$ such that $F(\bs{\hat{x}}_0)\in\mathcal{T}$ is the closest point to $\bs{x}_0$ on $\mathcal{T}$ (this is done with numerical optimization). We then decompose $\bs{x}_0$ as
\begin{align}\label{eq:step2}
\bs{x}_0 = F(\bs{\hat{x}}_0) + h\frac{\bs{\hat{n}}_0}{\vert\bs{\hat{n}}_0\vert}, \quad \bs{\hat{n}}_0 = \bs{\hat{n}}(\bs{\hat{x}}_0),
\end{align}
for some scalar $h$, which may be negative, obtained via
\begin{align}\label{eq:h}
h = (\bs{x}_0 - F(\bs{\hat{x}}_0)) \cdot \frac{\bs{\hat{n}}_0}{\vert\bs{\hat{n}}_0\vert}.
\end{align}

\paragraph{Step 3.~Taylor expanding/subtracting} We compute the strongly singular term in \cref{eq:step1},
\begin{align}\label{eq:Tn2}
T_{-2}(\bs{\hat{x}},h) = -\frac{h\varphi_0\vert\bs{\hat{n}}_0\vert}{\left[\vert J_0(\bs{\hat{x}} - \bs{\hat{x}}_0)\vert^2 + h^2\right]^{\frac{3}{2}}},
\end{align}
where $J_0$ is the Jacobian matrix at $\bs{\hat{x}}_0$ and $\varphi_0=\varphi(\bs{\hat{x}}_0)$, as well as the weakly singular term $T_{-1}$, whose expression will be derived in \cref{sec:step3}. We subtract them from/add them to \cref{eq:step1},
\begin{align}\label{eq:step3}
I(\bs{x}_0) = \, & \int_{\widehat{T}}\left[\frac{(F(\bs{\hat{x}})-\bs{x}_0)\cdot\bs{\hat{n}}(\bs{\hat{x}})}{\vert F(\bs{\hat{x}})-\bs{x}_0\vert^3}\varphi(\bs{\hat{x}}) - T_{-2}(\bs{\hat{x}},h) - T_{-1}(\bs{\hat{x}},h)\right]dS(\bs{\hat{x}}) \\
& + \int_{\widehat{T}}T_{-2}(\bs{\hat{x}},h)dS(\bs{\hat{x}}) + \int_{\widehat{T}}T_{-1}(\bs{\hat{x}},h)dS(\bs{\hat{x}}). \nonumber
\end{align}
The first integral is regularized---it may be computed with Gauss quadrature on triangles \cite{lether1976}. The last two integrals are singular or near-singular and will be computed in steps 4--5. 

\paragraph{Step 4.~Continuation approach} Let
\begin{align}\label{eq:In2-In1-2D}
I_{-2}(h) = \int_{\widehat{T}}T_{-2}(\bs{\hat{x}},h)dS(\bs{\hat{x}}), \quad I_{-1}(h) = \int_{\widehat{T}}T_{-1}(\bs{\hat{x}},h)dS(\bs{\hat{x}}).
\end{align}
Since $T_{-2}$ and $T_{-1}$ are homogeneous in $\bs{\hat{x}}$ and $h$, using the continuation approach~\cite{rosen1995}, we reduce the 2D integrals in \cref{eq:In2-In1-2D} to a sum of three 1D integrals along the edges of the shifted triangle $\widehat{T}-\bs{\hat{x}}_0$. For instance, for $I_{-2}$, this yields
\begin{align}\label{eq:In2-1D}
I_{-2}(h) = -\mrm{sign}(h)\varphi_0\vert\bs{\hat{n}}_0\vert\sum_{j=1}^3\hat{s}_j\int_{\partial\widehat{T}_j-\bs{\hat{x}}_0}\frac{\sqrt{\nrm^2 + h^2}-\vert h\vert}{\nrm^2\sqrt{\nrm^2 + h^2}}ds(\bs{\hat{x}}),
\end{align}
where the $\hat{s}_j$'s are the distances from the origin to the edges of $\widehat{T}-\bs{\hat{x}}_0$; see \cref{fig:shifted-tri}. We will provide the rigorous derivation of \cref{eq:In2-1D} in \cref{sec:step4}, as well as a formula for $I_{-1}$. We emphasize that \cref{eq:In2-1D} is equivalent to the solid angle formula \cref{eq:solid-angle}. The advantages of \cref{eq:In2-1D} are that it is applicable to other types of boundary elements, including quadrilateral elements, and that the methodology to derive it extends to elasticity potentials; see \cref{sec:elasticity}. There is a price to pay, however, as the formula requires the computation of 1D integrals. Nevertheless, since these integrals may be efficiently and exponentially accurately evaluated, their computational cost is negligible.

\begin{figure}
\centering
\begin{tikzpicture}[scale = 0.80]
      \draw[thick] (2*0, 2*0) -- (2*0, 2*2) {};
      \draw[thick] (2*0, 2*0) -- (2*2, 2*0) {};
      \draw[thick] (2*0, 2*2) -- (2*2, 2*0) {};
      \node[fill, circle, scale=0.5] at (2*0, 2*0) {};
      \node[fill, circle, scale=0.5] at (2*2, 2*0) {};
      \node[fill, circle, scale=0.5] at (2*0, 2*2) {};
      \node[fill, circle, scale=0.5] at (2*2/3, 2*2/3) {};
      \node[anchor=west] at (2*2/3+0.1, 2*2/3-0.1) {$\bs{0}$};
      \node[anchor=west] at (0.7, 0.5) {$\hat{s}_1$};
      \node[anchor=west] at (1+0.1, 2) {$\hat{s}_2$};
      \node[anchor=west] at (0.25, 1.65) {$\hat{s}_3$};
      \draw[thick, dashed] (2*2/3, 2*2/3) -- (2, 2) {};
      \draw[thick, dashed] (2*2/3, 2*2/3) -- (0, 1+1/3) {};
      \draw[thick, dashed] (2*2/3, 2*2/3) -- (1+1/3, 0) {};
      \draw[->, thick, red] (0, 4)--(-1.41421356237, 4);      
      \draw[->, thick, red] (0, 0)--(0,-1.41421356237);      
  	  \draw[->, thick, red] (4, 0)--(5, 1);      
  	  \node[anchor=west, red] at (0+0.05,-1.41421356237) {$\bs{\hat{\nu}}_1$};
  	  \node[anchor=east, red] at (5-0.1, 1) {$\bs{\hat{\nu}}_2$};
  	  \node[anchor=north, red] at (-1.41421356237, 4-0.05) {$\bs{\hat{\nu}}_3$};
      \node[anchor=east] at (0-0.1, 0-0.1) {$\bs{\hat{a}}_1-\bs{\hat{x}}_0=(-\hat{x}_0,-\hat{y}_0)$};
      \node[anchor=west] at (4+0.1, 0-0.1) {$\bs{\hat{a}}_2-\bs{\hat{x}}_0=(1-\hat{x}_0,-\hat{y}_0)$};
      \node[anchor=west] at (0+0.1, 4+0.1) {$\bs{\hat{a}}_3-\bs{\hat{x}}_0=(-\hat{x}_0,1-\hat{y}_0)$};
      \node[anchor=north] at (2, 0-0.1) {$\partial\widehat{T}_1-\bs{\hat{x}}_0$};
      \node[anchor=west] at (2.4+0.15, 1.6+0.15) {$\partial\widehat{T}_2-\bs{\hat{x}}_0$};
      \node[anchor=east] at (0-0.1, 2) {$\partial\widehat{T}_3-\bs{\hat{x}}_0$};
\end{tikzpicture}
\caption{\textit{The triangle $\widehat{T}-\bs{\hat{x}}_0$ above is the reference triangle $\widehat{T}$ of \cref{fig:planar-tri} shifted by $\bs{\hat{x}}_0=(\hat{x}_0,\hat{y}_0)$. The (signed) distances $\hat{s}_j$ to the edges are given by $\hat{s}_1=\hat{y}_0$, $\hat{s}_2=\sqrt{2}/2(1-\hat{x}_0-\hat{y}_0)$, and $\hat{s}_3=\hat{x}_0$.}}
\label{fig:shifted-tri}
\end{figure}
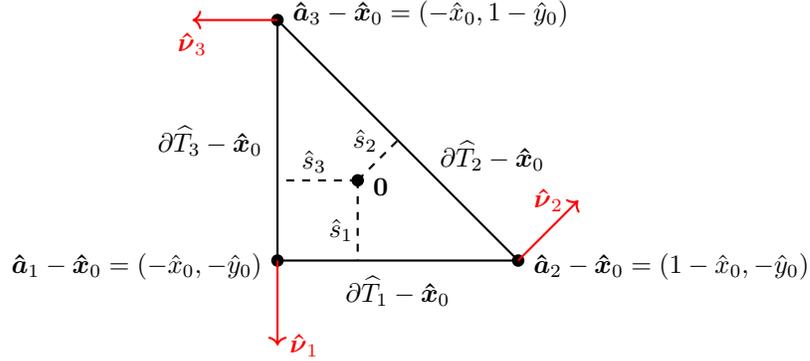

\paragraph{Step 5.~Transplanted Gauss quadrature} To circumvent near-singularity issues in the computation of $I_{-2}$ and $I_{-1}$, we employ transplanted quadrature \cite{hale2008} when the singularity is close to the edges. As in the companion paper \cite{montanelli2022}, we take advantage of the a priori knowledge of the singularities to utilize transplanted rules with significantly improved convergence rates \cite[Thm.~2.1]{montanelli2022}.

We now give more details about steps 3 and 4.

\subsection{Taylor expanding/subtracting (step 3)}\label{sec:step3}

Let $\rho$ be the diameter of the triangle $\mathcal{T}$, defined as the largest distance between two points on $\mathcal{T}$. We assume that the parameter of near-singularity $h$ in \cref{eq:h} is much smaller than $\rho$. (When $h\sim\rho$, the integral is analytic and Gauss quadrature converges exponentially---there is no need for any of this.) To compute $T_{-2}$ and $T_{-1}$, we first derive an expansion for $R=\vert F(\bs{\hat{x}})-\bs{x}_0\vert$. Let $\delta\hat{\bs{x}}=\bs{\hat{x}}-\bs{\hat{x}}_0=(\delta\hat{x}_1,\delta\hat{x}_2)$ and $\delta\hat{x}=\vert\delta\bs{\hat{x}}\vert$. Then,
\begin{align}\label{eq:Rn3}
R^{-3} = R_1^{-3} - \frac{3}{2}hA_2R_1^{-5} - \frac{3}{2}C_3R_1^{-5} + \OO(\delta\hat{x}^{-1}),
\end{align}
with $R_1=\sqrt{\vert J_0\delta\bs{\hat{x}}\vert^2 + h^2}=\OO(\delta\hat{x})$, so that $R_1^{-3} = \OO(\delta\hat{x}^{-3})$, and
\begin{align}
hA_2 = h\sum_{i=1}^3a_i\delta\hat{x}_1^{3-i}\delta\hat{x}_2^{i-1},
\quad C_3 = \sum_{i=1}^4c_i\delta \hat{x}_1^{4-i}\delta\hat{x}_2^{i-1}.
\end{align}
Since $h\ll\rho$, we have that $hA_2\ll R_1^2$, which implies that $hA_2R_1^{-5}$ and $C_3R_1^{-5}$ both are $\OO(\delta\hat{x}^{-2})$. (In the equations above and in what follows, the subscript $i$ in variables like $A_2$ and $C_3$ means that it is $\OO(\delta\hat{x}^i)$.) The formula \cref{eq:Rn3} regroups $r$-homogeneous functions in $(\delta\bs{\hat{x}},h)$ for increasing values of $r$ ($r=-3$ for $R_1^{-3}$ and $r=-2$ for the next two); the $a_i$'s and $c_i$'s are given in \cite[App.~A]{montanelli2022}. The integrand in \cref{eq:step1} that we wish to regularize may be written as $VR^{-3}$ with
\begin{align}
V(\bs{\hat{x}}) = (F(\bs{\hat{x}}) - F(\bs{\hat{x}}_0) - h\bs{\hat{n}}_0/\vert\bs{\hat{n}}_0\vert)\cdot\varphi(\bs{\hat{x}})\bs{\hat{n}}(\bs{\hat{x}}) = P(\bs{\hat{x}}) \cdot Q(\bs{\hat{x}}).
\end{align}
We write the Taylor series of $V$ as $V(\bs{\hat{x}}) = V_0 + V_1 + V_2(\bs{\hat{x}}) + \OO(\delta\hat{x}^3)$ with $V_0=0=\OO(1)$ and
\begin{align}\label{eq:V}
V_1 = -h\varphi_0\vert\bs{\hat{n}}_0\vert = \OO(\delta\hat{x}^1),
\quad V_2(\bs{\hat{x}}) = P_2(\bs{\hat{x}})\cdot Q_0 + P_1(\bs{\hat{x}})\cdot Q_1(\bs{\hat{x}}) = \OO(\delta\hat{x}^2).
\end{align}
The functions in \cref{eq:V} are $Q_0 = \varphi_0\bs{\hat{n}}_0$, $P_1(\bs{\hat{x}}) = J_0(\bs{\hat{x}} - \bs{\hat{x}}_0) - h\bs{\hat{n}}_0/\vert\bs{\hat{n}}_0\vert$, and
\begin{align}
& P_2(\bs{\hat{x}}) = \frac{1}{2}F_{\hat{x}_1\hat{x}_1}\delta\hat{x}_1^2 + \frac{1}{2}F_{\hat{x}_2\hat{x}_2}\delta\hat{x}_2^2 + F_{\hat{x}_1\hat{x}_2}\delta\hat{x}_1\delta\hat{x}_2, \\
& Q_1(\bs{\hat{x}}) = \left[\varphi_{\hat{x}_1}\delta\hat{x}_1+\varphi_{\hat{x}_2}\delta\hat{x}_2\right]\bs{\hat{n}}_0 + \varphi_0\left[\bs{\hat{n}}_{\hat{x}_1}\delta\hat{x}_1+\bs{\hat{n}}_{\hat{x}_2}\delta\hat{x}_2\right]. \nonumber
\end{align}
(The derivatives are evaluated at $\bs{\hat{x}}_0$.) Note that $\bs{\hat{n}}(\hat{\bs{x}}) = F_{\hat{x}_1}(\bs{\hat{x}})\times F_{\hat{x}_2}(\bs{\hat{x}})$ and hence
\begin{align}
& \bs{\hat{n}}_{\hat{x}_1}(\hat{\bs{x}}) = F_{\hat{x}_1\hat{x}_1}(\bs{\hat{x}})\times F_{\hat{x}_2}(\bs{\hat{x}}) + F_{\hat{x}_1}(\bs{\hat{x}})\times F_{\hat{x}_1\hat{x}_2}(\bs{\hat{x}}), \\
& \bs{\hat{n}}_{\hat{x}_2}(\hat{\bs{x}}) = F_{\hat{x}_1\hat{x}_2}(\bs{\hat{x}})\times F_{\hat{x}_2}(\bs{\hat{x}}) + F_{\hat{x}_1}(\bs{\hat{x}})\times F_{\hat{x}_2\hat{x}_2}(\bs{\hat{x}}). \nonumber
\end{align}
We conclude by writing $T_{-2} = V_1R_1^{-3}$ and $T_{-1} = V_2R_1^{-3} - \frac{3}{2}hV_1A_2R_1^{-5} - \frac{3}{2}V_1C_3R_1^{-5}$, i.e.,
\begin{align}\label{eq:Tn1}
T_{-1}(\bs{\hat{x}},h) = \, & \frac{P_2(\bs{\hat{x}})\cdot\varphi_0\bs{\hat{n}}_0 + \left[J_0(\bs{\hat{x}}-\bs{\hat{x}}_0)-h\bs{\hat{n}}_0/\vert\bs{\hat{n}}_0\vert\right]\cdot Q_1(\bs{\hat{x}})}{\left[\vert J_0(\bs{\hat{x}} - \bs{\hat{x}}_0)\vert^2 + h^2\right]^{\frac{3}{2}}} \\
& +\frac{3}{2}\sum_{i=1}^3a_i\delta\hat{x}_1^{3-i}\delta\hat{x}_2^{i-1}\frac{h^2\varphi_0\vert\bs{\hat{n}}_0\vert}{\left[\vert J_0(\bs{\hat{x}} - \bs{\hat{x}}_0)\vert^2 + h^2\right]^{\frac{5}{2}}} \nonumber \\
& +\frac{3}{2}\sum_{i=1}^4c_i\delta \hat{x}_1^{4-i}\delta\hat{x}_2^{i-1}\frac{h\varphi_0\vert\bs{\hat{n}}_0\vert}{\left[\vert J_0(\bs{\hat{x}} - \bs{\hat{x}}_0)\vert^2 + h^2\right]^{\frac{5}{2}}}. \nonumber
\end{align}
For planar triangles, the formula \cref{eq:Tn1} simplifies to
\begin{align}\label{eq:Tn1-planar}
T_{-1}(\bs{\hat{x}},h) = -h\vert\bs{\hat{n}}\vert\frac{\varphi_{\hat{x}_1}\delta\hat{x}_1+\varphi_{\hat{x}_2}\delta\hat{x}_2}{\left[\vert J(\bs{\hat{x}} - \bs{\hat{x}}_0)\vert^2 + h^2\right]^{\frac{3}{2}}} \quad\quad \text{(planar triangles)},
\end{align}
where the normal $\bs{\hat{n}}$ and the Jacobian $J$ are independent of $\bs{\hat{x}}_0$. We numerically validated the somewhat intricate formula \cref{eq:Tn1} and its simpler version \cref{eq:Tn1-planar} by plotting the regularizaition
\begin{align*}
\frac{(F(\bs{\hat{x}})-\bs{x}_0)\cdot\bs{\hat{n}}(\bs{\hat{x}})}{\vert F(\bs{\hat{x}})-\bs{x}_0\vert^3}\varphi(\bs{\hat{x}}) - T_{-2}(\bs{\hat{x}},h) - T_{-1}(\bs{\hat{x}},h)
\end{align*}
and checking that it remains bound throughout the domain for various values of $\bs{\hat{x}}_0$.

\subsection{The continuation approach (step 4)}\label{sec:step4}

We review the continuation approach for strong singularities, which yields the 1D formula \cref{eq:In2-1D} for $I_{-2}$. We also derive a 1D formula for $I_{-1}$.

Suppose $f$ is positive homogeneous in both $\bs{x}$ and $h$, i.e., there exists an integer $r$ such that $f(\lambda\bs{\hat{x}},\lambda h) = \lambda^r f(\bs{\hat{x}},h)$, for all $\bs{\hat{x}}$, $h\geq0$ and $\lambda>0$, and that we want to compute
\begin{align}
I(h) = \int_{\Omega}f(\bs{\hat{x}},h)dS(\bs{\hat{x}}),
\end{align}
for some bounded $\Omega$. The continuation approach yields the differential equation \cite[Eqn.~(2.3)]{rosen1995}
\begin{align}
hI'(h) - (r+2)I(h) = -\int_{\partial\Omega} f(\bs{\hat{x}},h)\bs{\hat{x}}\cdot\bs{\hat{\nu}}(\bs{\hat{x}})ds(\bs{\hat{x}}),
\end{align}
where $\bs{\hat{\nu}}(\bs{\hat{x}})$ is the normal along the boundary of $\Omega$. It can be integrated analytically to get
\begin{align}
I(h) = h^{r+2}\int_{\partial\Omega}\bs{\hat{x}}\cdot\bs{\hat{\nu}}(\bs{\hat{x}})\int_h^{\mrm{sign}(h)\infty}\frac{f(\bs{\hat{x}},u)}{u^{r+3}}du\,ds(\bs{\hat{x}}).
\end{align}

\paragraph{Strong singularities} For strong singularities we have $r=-2$, which leads to
\begin{align}\label{eq:I}
I(h) = \int_{\partial\Omega}\bs{\hat{x}}\cdot\bs{\hat{\nu}}(\bs{\hat{x}})\int_h^{\pm\infty}\frac{f(\bs{\hat{x}},u)}{u}du\,ds(\bs{\hat{x}}).
\end{align}
Let $F$ denote the indefinite integral
\begin{align}
F(\bs{\hat{x}}, h) = \int_0^h\frac{f(\bs{\hat{x}},u)}{u}du,
\end{align}
and write $F_{\pm\infty}(\bs{\hat{x}})=F(\bs{\hat{x}},\pm\infty)$. We also define the residues\footnote{Both residues are path independent, so long as the path encloses the singularity.} of $f$ and $F_{\pm\infty}$ via
\begin{align}
\mrm{Res}(f) = \int_{\partial\Omega} f(\bs{\hat{x}},0)\bs{\hat{x}}\cdot\bs{\hat{\nu}}(\bs{\hat{x}})ds(\bs{\hat{x}}), \quad
\mrm{Res}(F_{\pm\infty}) = \int_{\partial\Omega} F_{\pm\infty}(\bs{\hat{x}})\bs{\hat{x}}\cdot\bs{\hat{\nu}}(\bs{\hat{x}})ds(\bs{\hat{x}}).
\end{align}
The formula \cref{eq:I} can then be rewritten as
\begin{align}
I(h) = \int_{\partial\Omega}\left[F_{\pm\infty}(\bs{\hat{x}})-F(\bs{\hat{x}}, h)\right]\bs{\hat{x}}\cdot\bs{\hat{\nu}}(\bs{\hat{x}})ds(\bs{\hat{x}})
= \mrm{Res}(F_{\pm\infty}) - \int_{\partial\Omega}F(\bs{\hat{x}}, h)\bs{\hat{x}}\cdot\bs{\hat{\nu}}(\bs{\hat{x}})ds(\bs{\hat{x}}),
\end{align}
and the singular integral is
\begin{align}
I(0^{\pm}) = \lim_{h\to0^{\pm}}I(h) = \mrm{Res}(F_{\pm\infty}) - \lim_{h\to0^{\pm}}\int_{\partial\Omega}F(\bs{\hat{x}}, h)\bs{\hat{x}}\cdot\bs{\hat{\nu}}(\bs{\hat{x}})ds(\bs{\hat{x}}).
\end{align}
The values of $I(0^{\pm})$ can be either bounded or unbounded. Suppose that $f$ is of the form 
\begin{align}
f(\bs{\hat{x}}, h) = \frac{\hat{x}_1^{\ell_1}\hat{x}_2^{\ell_2}h^{\ell}}{(\vert\bs{\hat{x}}\vert^2 + h^2)^{\frac{m}{2}}}, \quad r = \ell_1 + \ell_2 + \ell - m = -2,
\end{align}
for some integers $\ell_1,\ell_2,\ell,m\geq0$. We have the following theorem \cite[Thm.~4]{rosen1995}.

\begin{theorem}[Continuation for strong singularities]\label{thm:continuation}
If $\ell=0$, then $F_{\pm\infty}=0$ and therefore
\begin{align}
I(h) = -\int_{\partial\Omega}F(\bs{\hat{x}}, h)\bs{\hat{x}}\cdot\bs{\hat{\nu}}(\bs{\hat{x}})ds(\bs{\hat{x}}).
\end{align}
If $\mrm{Res}(f)=0$, then $I(0)$ is bounded and coincides with the Cauchy principal value on the boundary,
\begin{align}
I(0) = -\lim_{h\to0}\int_{\partial\Omega}F(\bs{\hat{x}}, h)\bs{\hat{x}}\cdot\bs{\hat{\nu}}(\bs{\hat{x}})ds(\bs{\hat{x}}) = \int_{\partial\Omega}f(\bs{\hat{x}},0)\log\left(\vert\bs{\hat{x}}\vert\right)\bs{\hat{x}}\cdot\bs{\hat{\nu}}(\bs{\hat{x}})ds(\bs{\hat{x}}).
\end{align}
If $\mrm{Res}(f)\neq0$, then $I(0)$ is unbounded and the formula
\begin{align}
I(0) = -\lim_{h\to0}\int_{\partial\Omega}F(\bs{\hat{x}}, h)\bs{\hat{x}}\cdot\bs{\hat{\nu}}(\bs{\hat{x}})ds(\bs{\hat{x}})
\end{align}
will generate a finite part (the Cauchy principal value), as well as an infinite part with asymptotics of the form $\mrm{Res}(f)\log\vert h\vert$. In both cases, the integral is continuous at $h=0$.

If $\ell\neq0$ is odd, then $F_\infty\neq0$, and $\mrm{Res}(F_\infty) = -\mrm{Res}(F_{-\infty})\neq0$ in general. The integral reads
\begin{align}
I(h) = \pm\mrm{Res}(F_\infty) - \int_{\partial\Omega}F(\bs{\hat{x}}, h)\bs{\hat{x}}\cdot\bs{\hat{\nu}}(\bs{\hat{x}})ds(\bs{\hat{x}}),
\end{align}
with bounded value $I(0^{\pm}) = \pm\mrm{Res}(F_\infty)$, which is discontinuous at $h=0$. 

If $\ell\neq0$ is even, then $F_\infty\neq0$ but $\mrm{Res}(F_{\pm\infty})=0$. The integral is given by
\begin{align}
I(h) = - \int_{\partial\Omega}F(\bs{\hat{x}}, h)\bs{\hat{x}}\cdot\bs{\hat{\nu}}(\bs{\hat{x}})ds(\bs{\hat{x}}),
\end{align}
with bounded and continuous value $I(0)$.
\end{theorem}

\begin{proof}
See \cite[Thm.~4]{rosen1995} and \cite[sect.~3.3]{rosen1993}.
\end{proof}

We utilize the continuation approach to derive \cref{eq:In2-1D}. The integrand \cref{eq:Tn2} reads
\begin{align}\label{eq:regularization}
f(\bs{\hat{x}}, h) = -\frac{h\varphi_0\vert\bs{\hat{n}}_0\vert}{\left[\nrm^2 + h^2\right]^{\frac{3}{2}}},
\end{align}
after translation by $\bs{\hat{x}}_0$. We apply \cref{thm:continuation} with $\ell_1=\ell_2=0$, $\ell=1$, and $m=3$. We have
\begin{align}
F(\bs{\hat{x}}, h) = -\frac{h\varphi_0\vert\bs{\hat{n}}_0\vert}{\nrm^2\sqrt{\nrm^2 + h^2}}, \quad 
\mrm{Res}(F_\infty) = -\int_{\partial\widehat{T}-\bs{\hat{x}}_0}\frac{\varphi_0\vert\bs{\hat{n}}_0\vert}{\nrm^2}\bs{\hat{x}}\cdot\bs{\hat{\nu}}(\bs{\hat{x}})ds(\bs{\hat{x}}).
\end{align}
Therefore, we arrive at the following formula,
\begin{align}
I_{-2}(h) = \mp\int_{\partial\widehat{T}-\bs{\hat{x}}_0}\frac{\varphi_0\vert\bs{\hat{n}}_0\vert}{\nrm^2}\bs{\hat{x}}\cdot\bs{\hat{\nu}}(\bs{\hat{x}})ds(\bs{\hat{x}})
+ \int_{\partial\widehat{T}-\bs{\hat{x}}_0}\frac{h\varphi_0\vert\bs{\hat{n}}_0\vert}{\nrm^2\sqrt{\nrm^2 + h^2}}\bs{\hat{x}}\cdot\bs{\hat{\nu}}(\bs{\hat{x}})ds(\bs{\hat{x}}).
\end{align}
To obtain \cref{eq:In2-1D}, we write $h=\mrm{sign}(h)\vert h\vert$ and observe that $\bs{\hat{x}}\cdot\bs{\hat{\nu}}(\bs{\hat{x}})$ is constant on each edge $\partial\widehat{T}_j-\bs{\hat{x}}_0$ of the shifted triangle and equals the distance from the origin; see \cref{fig:shifted-tri}. Finally, we emphasize that \cref{thm:continuation} tells us that $I_{-2}(h)$ is discontinuous but remains bounded at $0^{\pm}$; the limiting values are those listed in \cref{tab:solid-angle-planar} (with a minus sign). In practice, we set $I_{-2}(0)=(I_{-2}(0^+)+I_{-2}(0^-))/2=0$ and $I_{-2}(h)=0$ for $\vert h\vert$ below some threshold. 

\paragraph{Weak singularities} To compute $I_{-1}$, we utilize to continuation approach as in \cite{montanelli2022}. This yields
\begin{align}\label{eq:In1}
I_{-1}(\bs{x}_0) = \, & \sum_{j=1}^3\hat{s}_j\int_{\partial\widehat{T}_j-\bs{\hat{x}}_0}\hspace{-0.5cm}\left[P_2(\bs{\hat{x}})\cdot\varphi_0\bs{\hat{n}}_0 + J_0\bs{\hat{x}}\cdot Q_1(\bs{\hat{x}})\right]T_{-1}^{p}(\bs{\hat{x}},h)ds(\bs{\hat{x}}) \\
&  - h\sum_{j=1}^3\hat{s}_j\int_{\partial\widehat{T}_j-\bs{\hat{x}}_0}\left[\bs{\hat{n}}_0/\vert\bs{\hat{n}}_0\vert\cdot Q_1(\bs{\hat{x}})\right]T_{-1}^{q}(\bs{\hat{x}},h)ds(\bs{\hat{x}}) \nonumber \\
& +\frac{3}{2}h\varphi_0\vert\bs{\hat{n}}_0\vert\sum_{i=1}^3a_i\sum_{j=1}^3\hat{s}_j\int_{\partial\widehat{T}_j-\bs{\hat{x}}_0}\hat{x}_1^{3-i}\hat{x}_2^{i-1}T_{-1}^a(\bs{\hat{x}},h)ds(\bs{\hat{x}}) \nonumber \\
& +\frac{3}{2}h\varphi_0\vert\bs{\hat{n}}_0\vert\sum_{i=1}^4c_i\sum_{j=1}^3\hat{s}_j\int_{\partial\widehat{T}_j-\bs{\hat{x}}_0}\hat{x}_1^{4-i}\hat{x}_2^{i-1}T_{-1}^c(\bs{\hat{x}},h)ds(\bs{\hat{x}}), \nonumber
\end{align}
with functions $T_{-1}^{p}$, $T_{-1}^{q}$, $T_{-1}^{a}$, and $T_{-1}^{c}$ given by
\begin{align}
& T_{-1}^{p}(\bs{\hat{x}},h) = \frac{\nrm^2 + 2h\left(h - \mrm{sign}(h)\sqrt{\nrm^2 + h^2}\right)}{\nrm^4\sqrt{\nrm^2 + h^2}}, \\
& T_{-1}^{q}(\bs{\hat{x}},h) = \frac{\sqrt{\nrm^2 + h^2}\mrm{arcsinh}\left(\frac{\nrm}{\vert h\vert}\right) - \nrm}{\nrm^3\sqrt{\nrm^2 + h^2}}, \nonumber \\
& T_{-1}^a(\bs{\hat{x}},h) = \frac{-2h^3 - 3h\nrm^2 + 2\,\mrm{sign}(h)\left[\nrm^2 + h^2\right]^{\frac{3}{2}}}{3\nrm^4\left[\nrm^2 + h^2\right]^{\frac{3}{2}}}, \nonumber \\
& T_{-1}^c(\bs{\hat{x}},h) = \frac{\frac{-3h^2\vert J_0\bs{\hat{x}}_0\vert - 4\nrm^3}{\left[\nrm^2 + h^2\right]^{\frac{3}{2}}} + 3\mrm{arcsinh}\left(\frac{\vert J_0\bs{\hat{x}}_0\vert}{\vert h\vert}\right)}{3\nrm^5}. \nonumber
\end{align}
For planar triangles, the formula \cref{eq:In1} simplifies to
\begin{align}
I_{-1}(\bs{x}_0) = -h\vert\bs{\hat{n}}\vert\sum_{j=1}^3\hat{s}_j\int_{\partial\widehat{T}_j-\bs{\hat{x}}_0}\left(\varphi_{\hat{x}_1}\delta\hat{x}_1+\varphi_{\hat{x}_2}\delta\hat{x}_2\right)T_{-1}^{q}(\bs{\hat{x}},h)ds(\bs{\hat{x}}) \quad\quad \text{(planar triangles)}.
\end{align}

\subsection{Computing integrals over two curved triangles}

The simplest example of an integral over two curved triangles is
\begin{align}\label{eq:integral-4D}
I = -\int_{\mathcal{T}}\Omega(\bs{y})dS(\bs{y}) = \int_{\mathcal{T}}\int_{\mathcal{T}}\frac{(\bs{x}-\bs{y})\cdot\bs{n}(\bs{x})}{\vert\bs{x}-\bs{y}\vert^3}dS(\bs{x})dS(\bs{y}),
\end{align}
where $\mathcal{T}$ is parametrized by some function $F$. We map the $\bs{y}$-integral back to $\widehat{T}$,
\begin{align}
I = \int_{\mathcal{T}}\int_{\widehat{T}}\frac{(\bs{x} - F(\bs{\hat{y}}))\cdot\bs{n}(\bs{x})}{\vert\bs{x} - F(\bs{\hat{y}})\vert^3}\vert\bs{\hat{n}}(\bs{\hat{y}})\vert dS(\bs{\hat{y}})dS(\bs{x}), \quad
\bs{\hat{n}}(\bs{\hat{y}}) = F_{\hat{y}_1}(\bs{\hat{y}})\times F_{\hat{y}_2}(\bs{\hat{y}}).
\end{align}
We then compute the $\bs{\hat{y}}$-integral with $N$-point Gauss quadrature on triangles \cite{lether1976},
\begin{align}
I \approx \sum_{n=1}^Nw_n\vert\bs{\hat{n}}(\bs{\hat{y}}_n)\vert I_n, \quad
I_n = \int_{\mathcal{T}}\frac{(\bs{x} - F(\bs{\hat{y}}_n))\cdot\bs{n}(\bs{x})}{\vert\bs{x} - F(\bs{\hat{y}}_n)\vert^3}dS(\bs{x}).
\end{align}
The $I_n$'s are computed with the method described in this paper with $N$ points; the total number of points is therefore $M=N^2$. We adopt the same approach when we integrate over two different curved triangles $\mathcal{T}$ and $\mathcal{T}'$. For planar triangles, exact formulas may be found in \cite{lenoir2012}.

\subsection{Assembling boundary element matrices}\label{sec:bem}

Discretizing \cref{eq:CBIE} with a boundary element method yields the computation of integrals of the form \cite[Chap.~5]{sauter2011}
\begin{align}
I = \int_{\mathcal{T}}\int_{\mathcal{T}}\frac{\partial G(\bs{x},\bs{y})}{\partial\bs{n}(\bs{y})}\varphi(F^{-1}(\bs{x}))\varphi(F^{-1}(\bs{y}))dS(\bs{x})dS(\bs{y}),
\end{align}
for some curved triangle $\mathcal{T}$ parametrized by a function $F$ of degree $q\geq1$ and for some basis function $\varphi$ of degree $p\geq0$. (For the sake of clarity and simplicity, our focus here is directed towards the strongly singular integrals that appear in \cref{eq:CBIE} and scenarios involving identical triangles and basis functions.) The conormal derivative of $G$ is given by
\begin{align}
\frac{\partial G(\bs{x},\bs{y})}{\partial\bs{n}(\bs{y})} = \frac{1}{4\pi}\left(1 - ik\vert\bs{x}-\bs{y}\vert\right)\frac{e^{ik\vert\bs{x}-\bs{y}\vert}}{\vert\bs{x}-\bs{y}\vert^3}(\bs{x}-\bs{y})\cdot\bs{n}(\bs{y}),
\end{align}
which we write as
\begin{align}
\frac{\partial G(\bs{x},\bs{y})}{\partial\bs{n}(\bs{y})} = \frac{1}{4\pi}\left[\frac{1}{\vert\bs{x}-\bs{y}\vert^3} + \frac{k^2}{2}\frac{1}{\vert\bs{x}-\bs{y}\vert} + S(\vert\bs{x}-\bs{y}\vert)\right](\bs{x}-\bs{y})\cdot\bs{n}(\bs{y}),
\end{align}
with a function $S(r)$ given by
\begin{align}
S(r) = \frac{-1 - k^2r^2/2 + e^{ikr}(1 - ikr)}{r^3}.
\end{align}
We first map the $\bs{x}$-integral back to $\widehat{T}$ and then discretize it with $N$-point Gauss quadrature,
\begin{align}
I \approx \sum_{n=1}^Nw_n\varphi(\bs{\hat{x}}_n)\vert\bs{\hat{n}}(\bs{\hat{x}}_n)\vert I_n, 
\end{align}
where the $I_n$'s are given by, for $\bs{y}_n=F(\bs{\hat{x}}_n)$,
\begin{align}
I_n = \frac{1}{4\pi}\int_{\mathcal{T}}\left[\frac{1}{\vert\bs{y}_n - \bs{y}\vert^3} + \frac{k^2}{2}\frac{1}{\vert\bs{y}_n - \bs{y}\vert} + S(\vert\bs{y}_n - \bs{y}\vert)\right](\bs{y}_n - \bs{y})\cdot\bs{n}(\bs{y})\varphi(F^{-1}(\bs{y}))dS(\bs{y}).
\end{align}
The first term is integrated with the method described in this paper with $N$ quadrature points---the convergence rate is $\OO(N^{-1})$, as we will see in \cref{sec:numerics}. The second and third terms are integrated with $N$-point Gauss quadrature on triangles---the convergence rate is $\OO(N^{-1})$ for the first one (it is bounded) and $\OO(N^{-1.5})$ for the second one (it has bounded first derivatives).

\section{Numerical examples}\label{sec:numerics} 

We present in this section several numerical experiments in 3D to demonstrate the capabilities of our algorithms.

\subsection{Computation of singular/near-singular integrals}

We start by testing our method for computing 2D and 4D singular/near-singular integrals.

\paragraph{Singluar/near-singular integrals over a triangle} Consider the quadratic triangle $\mathcal{T}$ given by
\begin{align}\label{eq:quad-tri}
& \bs{a}_1 = (0,0,0), \quad\quad \bs{a}_4 = (1/2,0,0), \\
& \bs{a}_2 = (1,0,0), \quad\quad \bs{a}_5 = (a,b,c), \nonumber \\
& \bs{a}_3 = (0,1,0), \quad\quad \bs{a}_6 = (0,1/2,0), \nonumber
\end{align}
for some scalars $a$, $b$, and $c$; see \cref{fig:exp-triangle} (left). The mapping $F$ and its Jacobian matrix $J$ read
\begin{align}\label{eq:quad-tri2}
F(\bs{\hat{x}}) = \begin{pmatrix}
\hat{x}_1 + 2(2a-1)\hat{x}_1\hat{x}_2 \eqvsp
\hat{x}_2 + 2(2b-1)\hat{x}_1\hat{x}_2 \eqvsp
4c\hat{x}_1\hat{x}_2
\end{pmatrix}, \quad
J(\bs{\hat{x}}) = \begin{pmatrix}
1 + 2(2a-1)\hat{x}_2 & 2(2a-1)\hat{x}_1 \eqvsp
2(2b-1)\hat{x}_2 & 1 + 2(2b-1)\hat{x}_1 \eqvsp
4c\hat{x}_2 & 4c\hat{x}_1
\end{pmatrix}.
\end{align}
We take $a=0.6$, $b=0.7$, and $c=0.5$, and compute the following integral for different values of $\bs{x}_0$,
\begin{align}\label{eq:exp-double-layer}
I(\bs{x}_0) = \int_\mathcal{T}\frac{(\bs{x} - \bs{x}_0)\cdot\bs{n}(\bs{x})}{\vert\bs{x}-\bs{x}_0\vert^3}dS(\bs{x}).
\end{align}
We report the results for $\bs{x}_0=F(0.2,0.4)$ and $\bs{x}_0=F(0.2,0.4)-10^{-4}\bs{z}$ with $\bs{z}=(0,0,1)$ in \cref{fig:exp-double-layer}. The ``exact'' values are computed in Mathematica to $12$-digit accuracy, while the numerical values are computed with $N=n^2$ points for the 2D integrals and $n$ points in 1D; we take $2\leq n\leq200$. (The errors come from the 2D integrals---the 1D errors are much smaller.) In the singular case, the integrand is only weakly singular. Therefore, regularizing with $T_{-2}$ does not accelerate convergence, while $T_{-1}$ regularizaiton yields linear convergence, which is consistent with our previous work \cite{montanelli2022}. In the nearly singular case, the kernel is ``numerically'' strongly singular, and hence $T_{-2}$ regulariztion is needed to get $\OO(N^{-0.5})$ convergence. Again, regularizing with $T_{-1}$ leads to linear convergence.

\begin{figure}
\def\scl{0.2}
\centering
\includegraphics[scale=\scl]{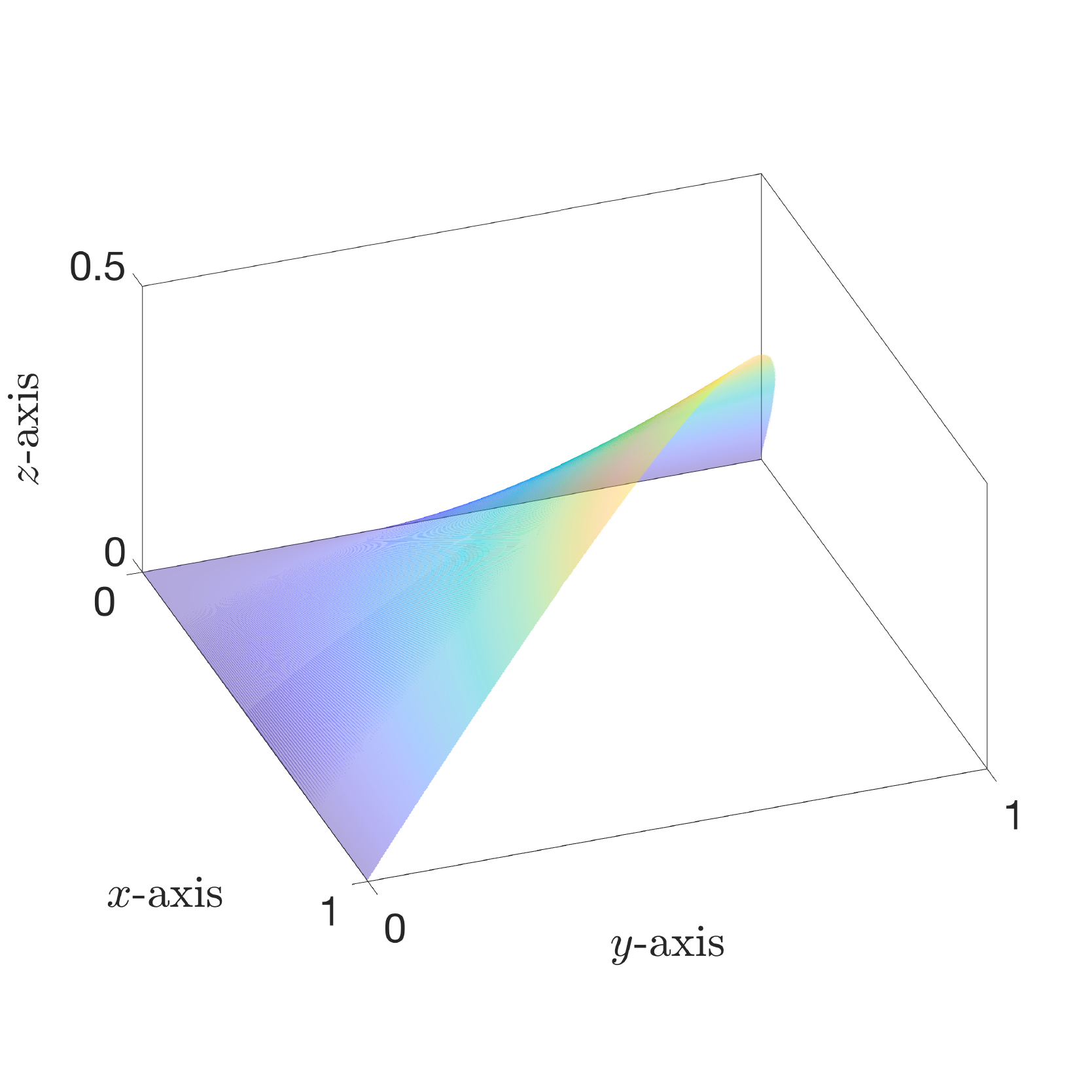}
\includegraphics[scale=\scl]{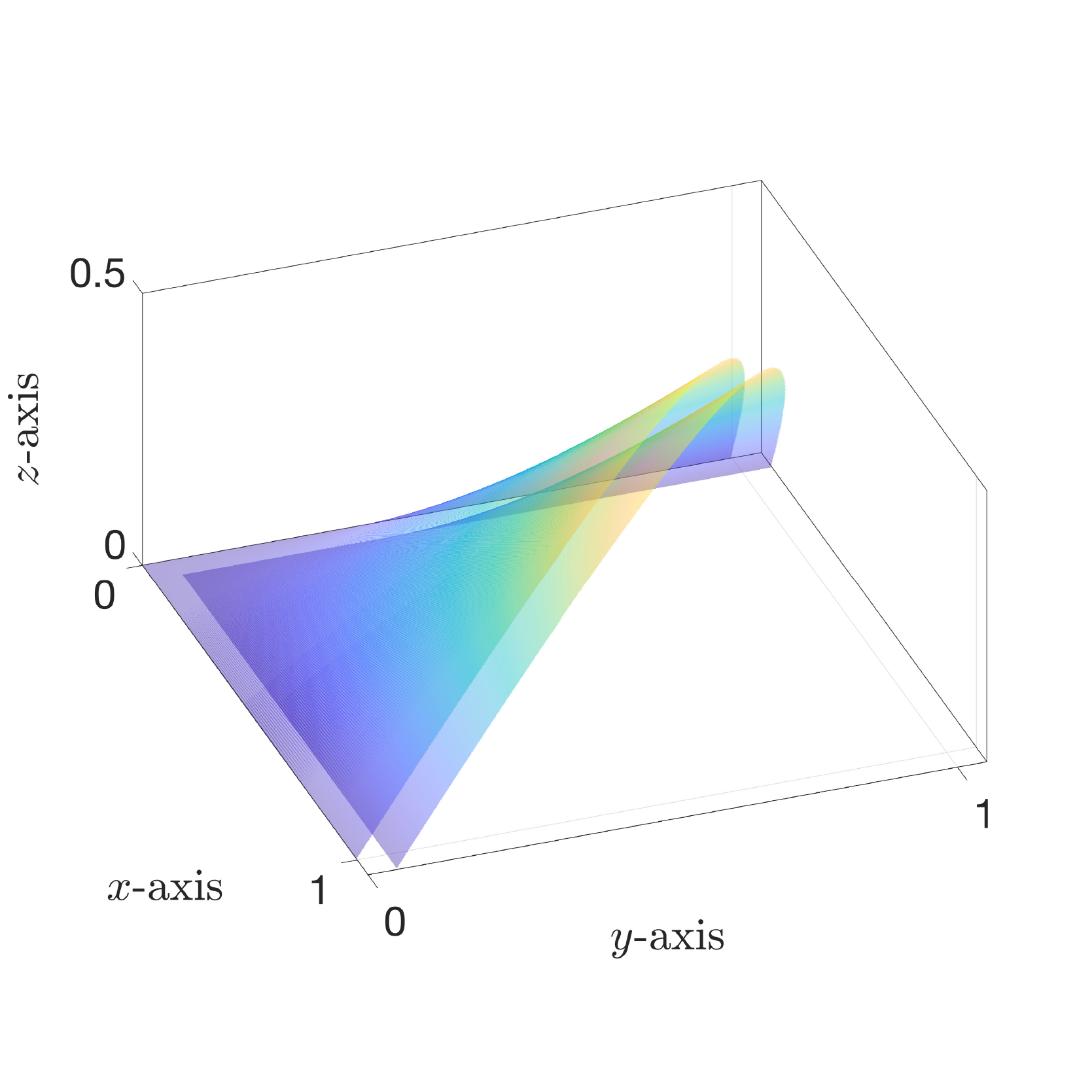}
\caption{\textit{We demonstrate our algorithms on the quadratic triangle $\mathcal{T}$ displayed on the left, which is defined by \cref{eq:quad-tri}--\cref{eq:quad-tri2} with parameters $a=0.6$, $b=0.7$, and $c=0.5$. For 4D integrals, we integrate over the two quadratic triangles $\mathcal{T}$ and $\mathcal{T}'=\mathcal{T}+5\times10^{-2}(1,1,0)$ shown on the right.}}
\label{fig:exp-triangle}
\end{figure}

\begin{figure}
\def\scl{0.425}
\centering
\includegraphics[scale=\scl]{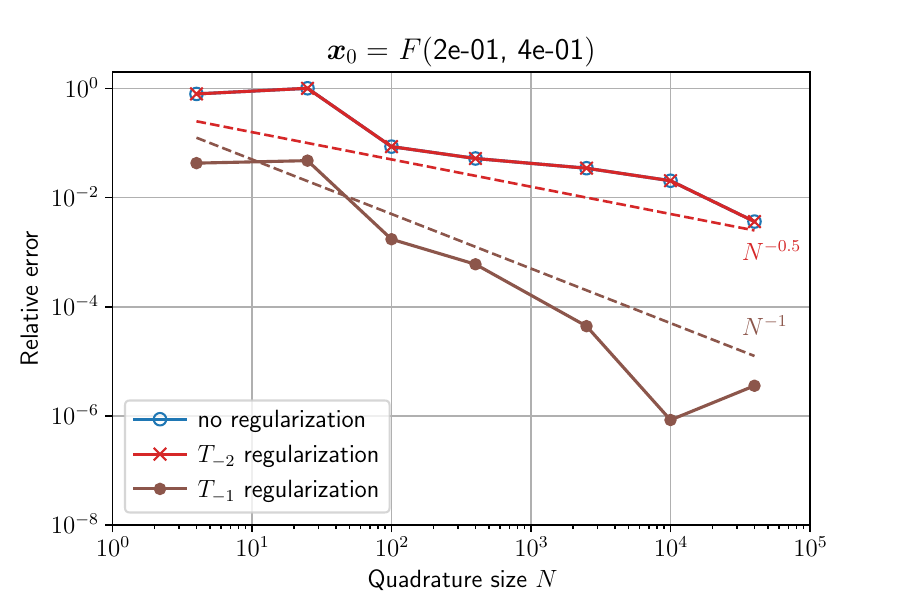}
\includegraphics[scale=\scl]{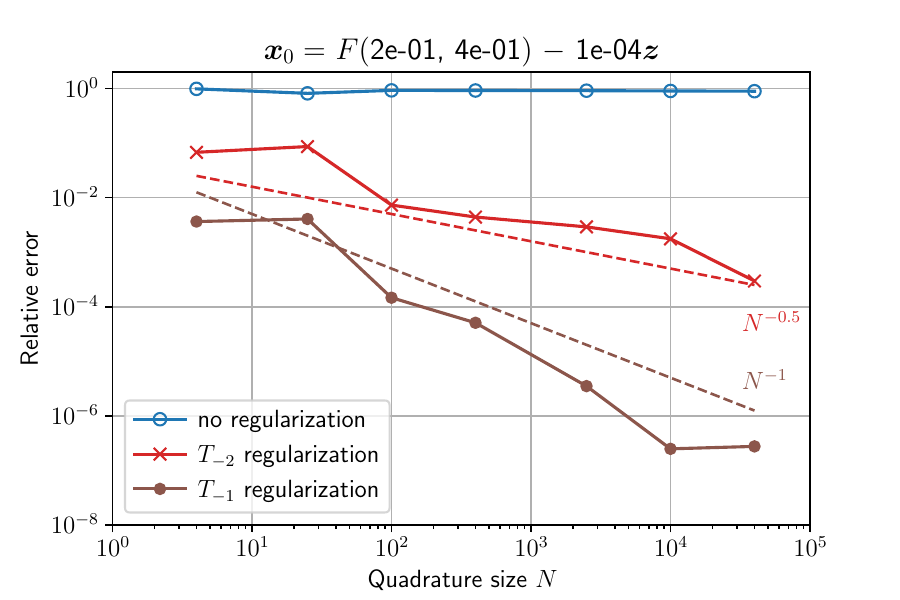}
\caption{\textit{When the singularity is exactly on the triangle, cancellation with the normal leads to a weakly singular kernel. Regularization is not needed to obtain convergence at a rate $\OO(N^{-0.5})$, while $T_{-1}$ regularization accelerates the convergence to linear (left). The nearly singular case is harder---the integrand is numerically strongly singular and regularization is needed to converge (right).}}
\label{fig:exp-double-layer}
\end{figure}

\paragraph{Singluar/near-singular integrals over two triangles} We first compute 4D singular integrals of the form of \cref{eq:integral-4D} by integrating twice over the triangle $\mathcal{T}$ defined by \cref{eq:quad-tri}--\cref{eq:quad-tri2}. We combine our method for the $\bs{x}$-integral with Gauss quadrature for the $\bs{y}$-integral, and report the results in \cref{fig:exp-double-layer-4D} (left). The ``exact'' value is computed with the method of Sauter and Schwab to 8-digit accuracy \cite[Chap.~5]{sauter2011}. We observe convergence at a rate $\OO(M^{-0.5})$ for $T_{-1}$ regularization, where $M$ is the total number of points in 4D. This is consistent with our previous work \cite{montanelli2022}. To understand this convergence rate, we note that quadrature for the $\bs{x}$-integral converges at a rate $\OO(N^{-1})$ using $T_{-1}$ regularization, while quadrature for the $\bs{y}$-integral converges at a rate at least $\OO(N^{-1})$ since the solid angle $\Omega(\bs{y})$ is as smooth as the regularized $\bs{x}$-integrand (they are both of bounded variation), as illustrated in \cref{fig:solid-angle} (left)---the 4D convergence rate is therefore $\OO(M^{-0.5})$.

We consider now the case where we integrate over $\bs{x}\in\mathcal{T}$, where $\mathcal{T}$ is given by \cref{eq:quad-tri}--\cref{eq:quad-tri2}, and $\bs{y}\in\mathcal{T}'=\mathcal{T}+5\times10^{-2}(1,1,0)$; see \cref{fig:exp-triangle} (right). We report the results in \cref{fig:exp-double-layer-4D} (right). The ``exact'' value is computed in Mathematica to $6$-digit accuracy. For this example, the method of Sauter and Schwab fails to provide accurate results as it simply utilizes 4D Gauss quadrature, which corresponds to the case with no regularization. (Their method is tailored to the cases where the triangles are identical, or share an edge or a vertex.) Our method converges slightly faster than in the singular case. We emphasize that the solid angle $\Omega(\bs{y})=\Omega_{-2}(\bs{y})+\Omega_{-1}(\bs{y})$ may have sharp variations since $\Omega_{-2}(\bs{y})$ quickly varies between values around $2\pi$ and $0$ when $\bs{y}\in\mathcal{T}'$ lies below the interior or the exterior of $\mathcal{T}$, as seen in \cref{fig:solid-angle} (right)---however, it is still at least as smooth as the regularized $\bs{x}$-integrand, so it does not affect the 4D convergence rate.

\begin{figure}
\def\scl{0.425}
\centering
\includegraphics[scale=\scl]{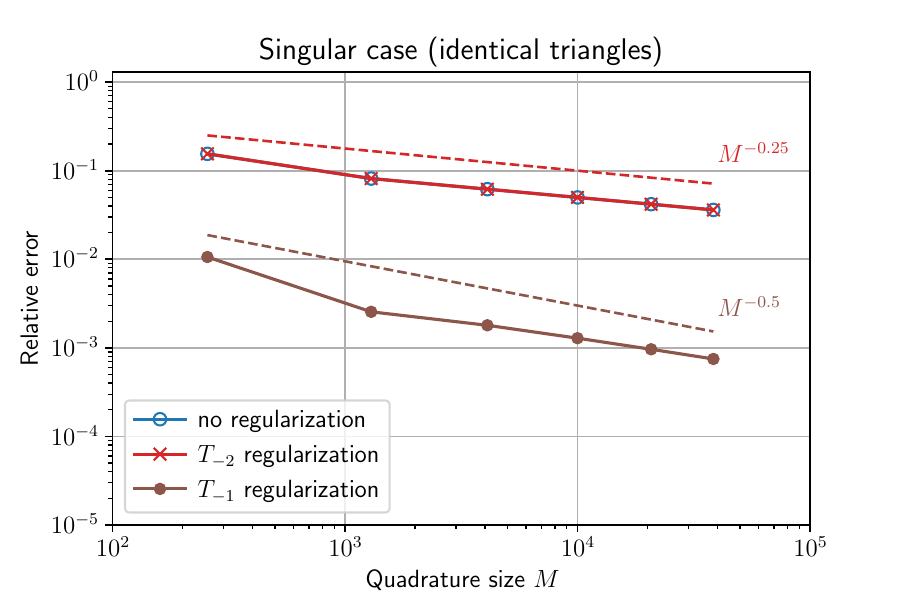}
\includegraphics[scale=\scl]{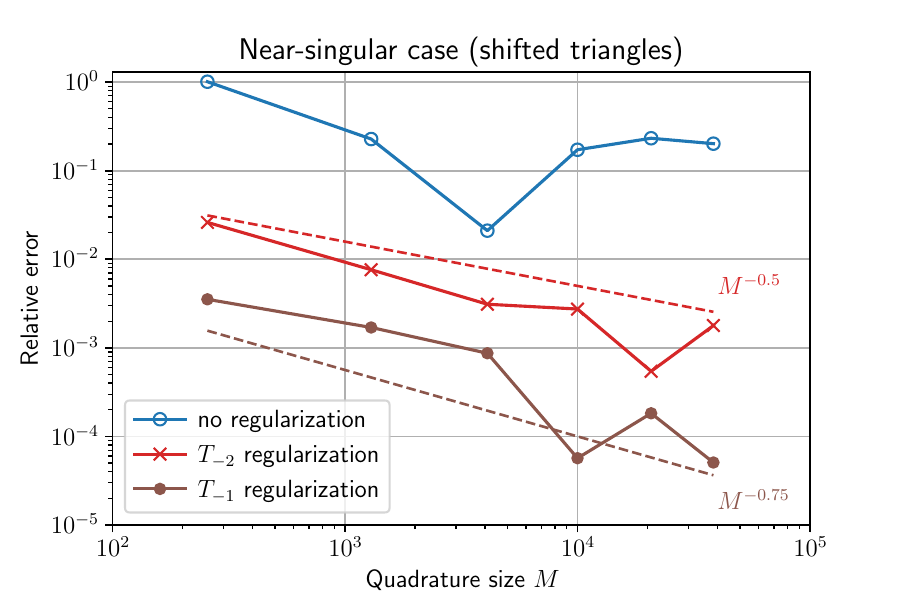}
\caption{\textit{Combining the method described in this paper with Gauss quadrature on triangles yields a method for computing 4D singular integrals of the form of \cref{eq:integral-4D}. The method converges at a speed $\OO(M^{-0.5})$ with the total 4D number of points $M=N^2$ when using $T_{-1}$-regularization.}}
\label{fig:exp-double-layer-4D}
\end{figure}

\begin{figure}
\def\scl{0.2}
\centering
\includegraphics[scale=\scl]{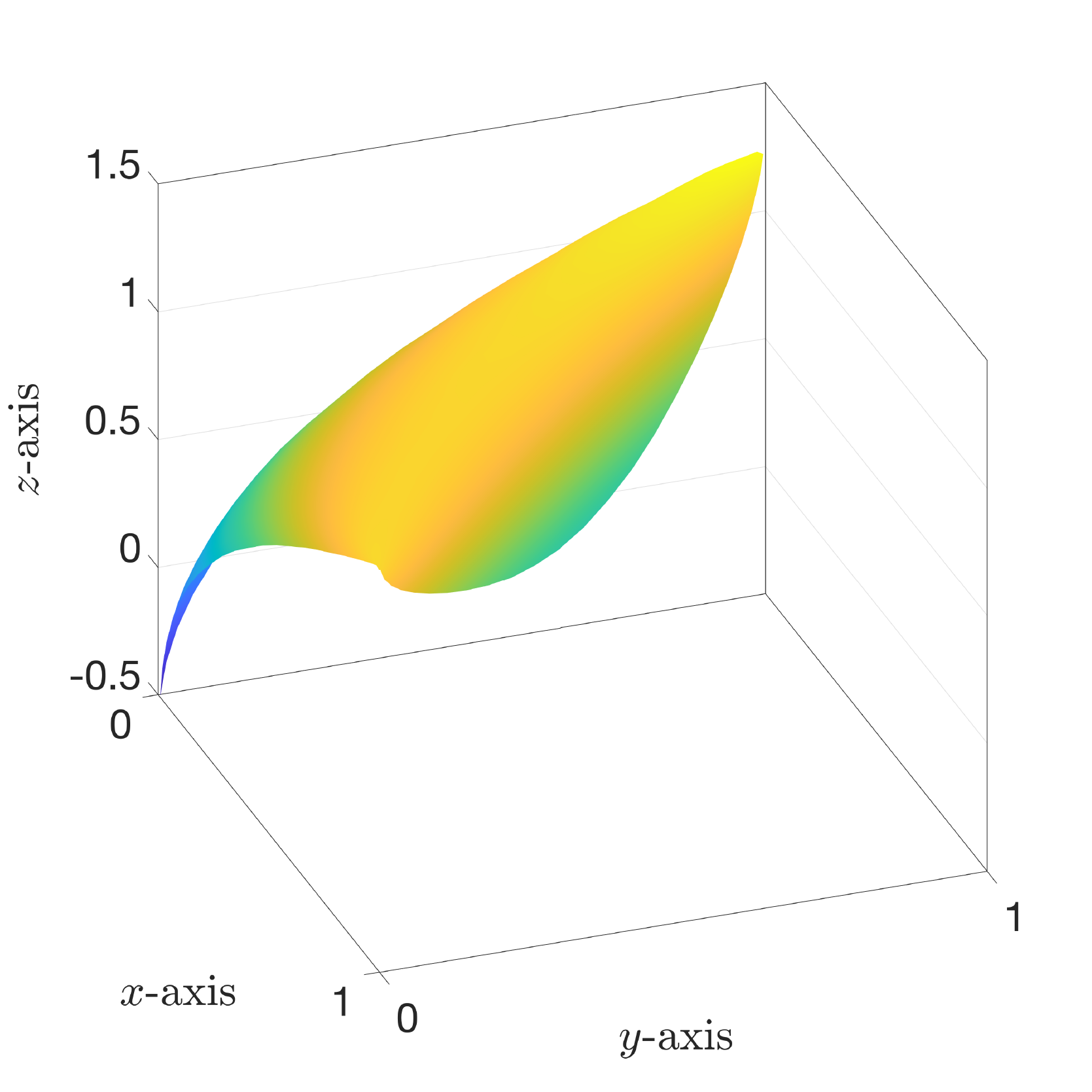}
\hspace{.2cm}
\includegraphics[scale=\scl]{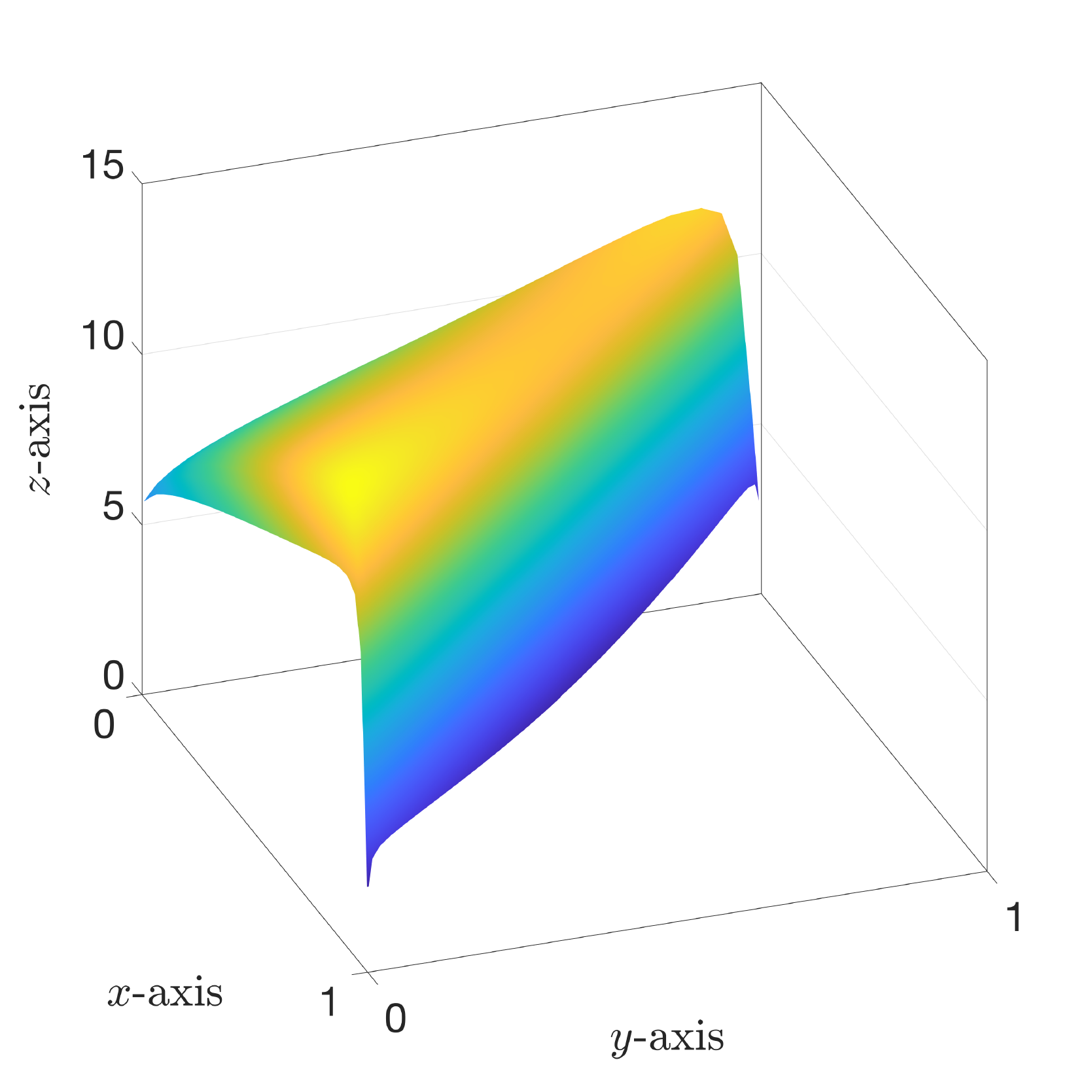}
\caption{\textit{Solid angle function $\Omega(\bs{y})$ for identical (left) and shifted triangles (right). Despite the sharp variations near the edges, the solid angle function $\Omega(\bs{y})$ is of bounded variation.}}
\label{fig:solid-angle}
\end{figure}

\paragraph{Singluar/near-singular integrals over two triangles (spherical mesh)} We consider a mesh of 84 quadratic triangles, which we generated with Gmsh \cite{geuzaine2009}. (This corresponds to the coarser mesh we will use in \cref{sec:BEM} for boundary element computations.) We compute 4D singular integrals of the form of \cref{eq:integral-4D} in three different scenarios: identical triangles, triangles sharing an edge, and triangles sharing a vertex. For each scenario, the reference value is computed to $8$-digit accuracy with the method of Sauter and Schwab. We report the results in \cref{fig:exp-double-layer-4D-sphere}.

\begin{figure}
\def\scl{0.175}
\def\sclb{0.425}
\centering
\includegraphics[scale=\scl]{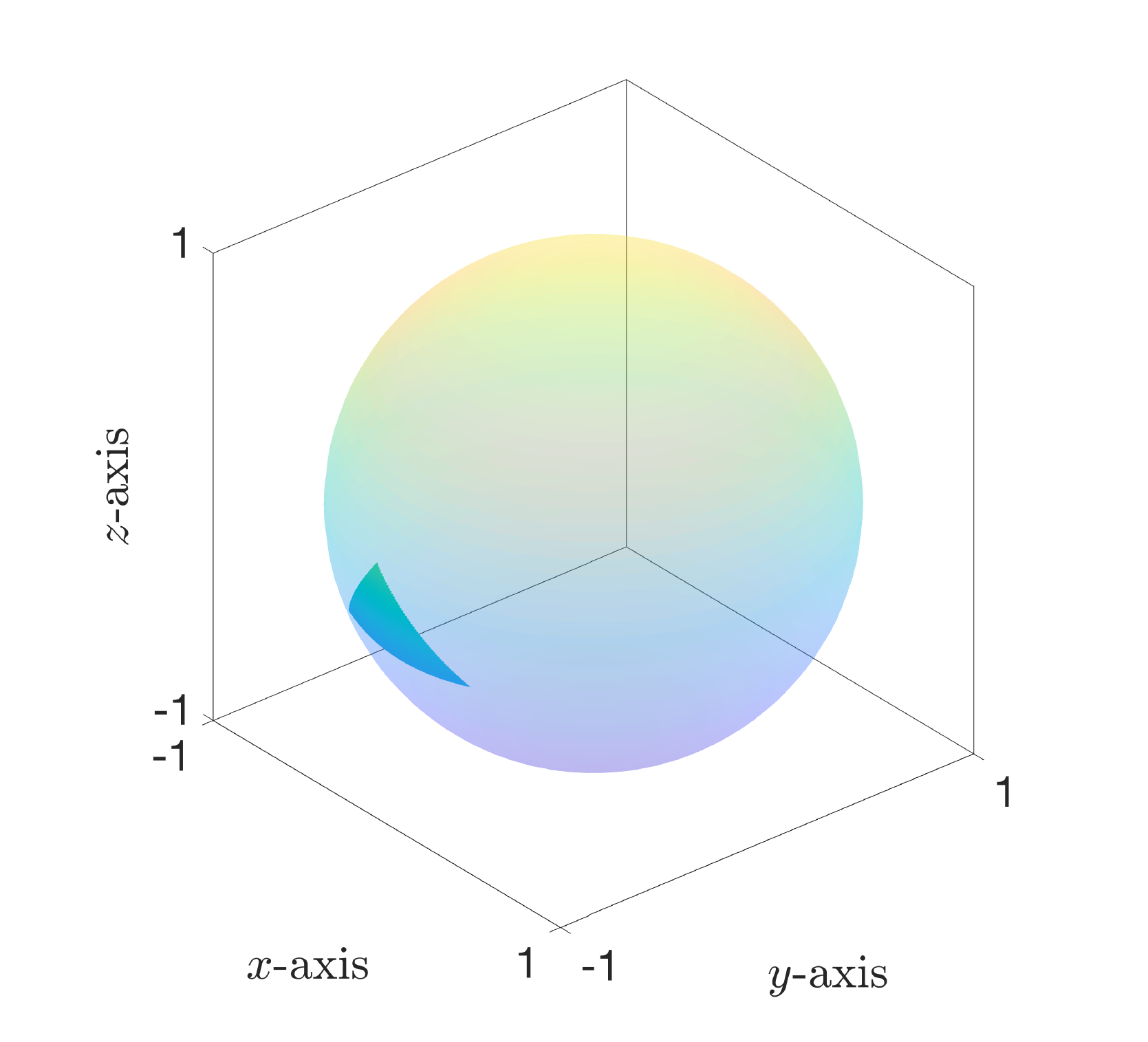}
\includegraphics[scale=\sclb]{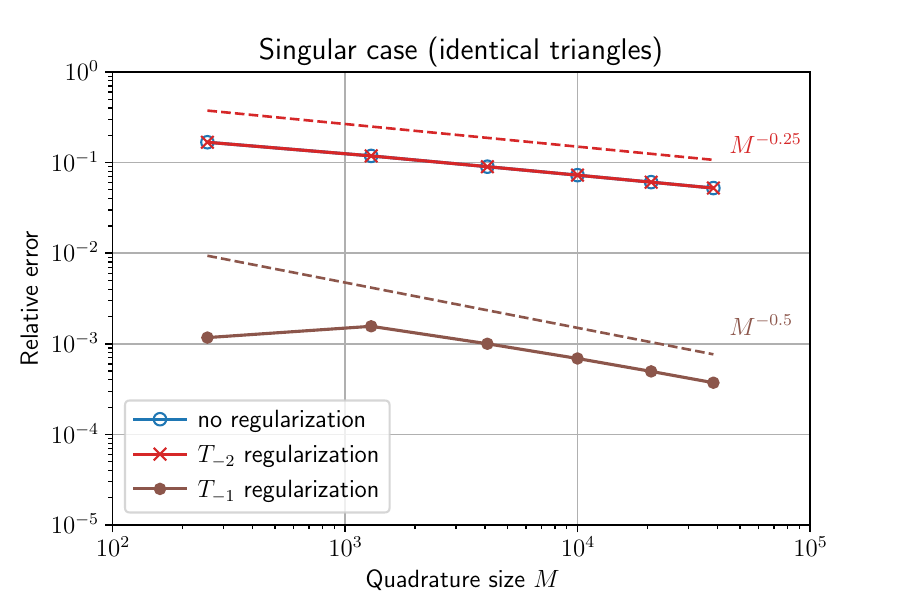}
\includegraphics[scale=\scl]{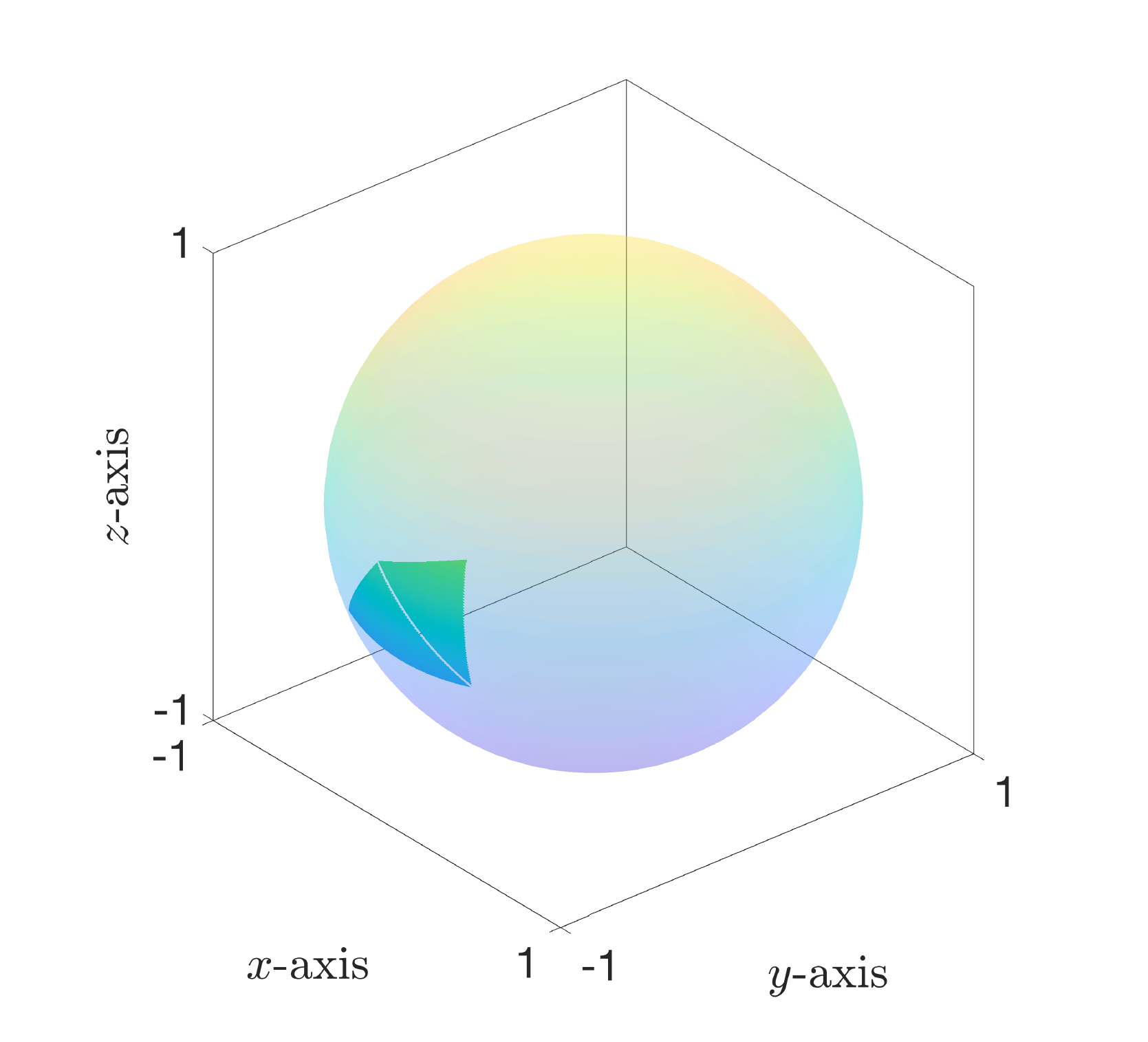}
\includegraphics[scale=\sclb]{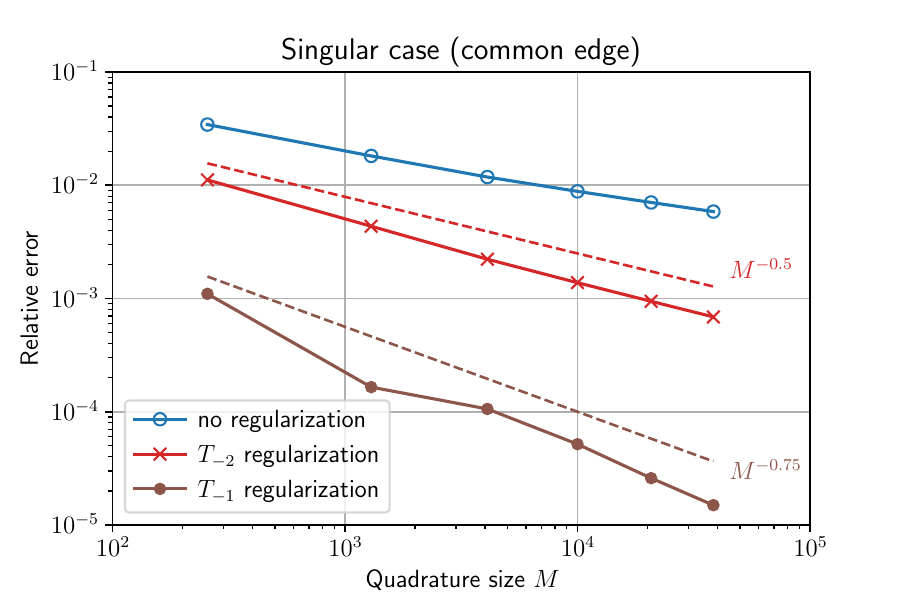}
\includegraphics[scale=\scl]{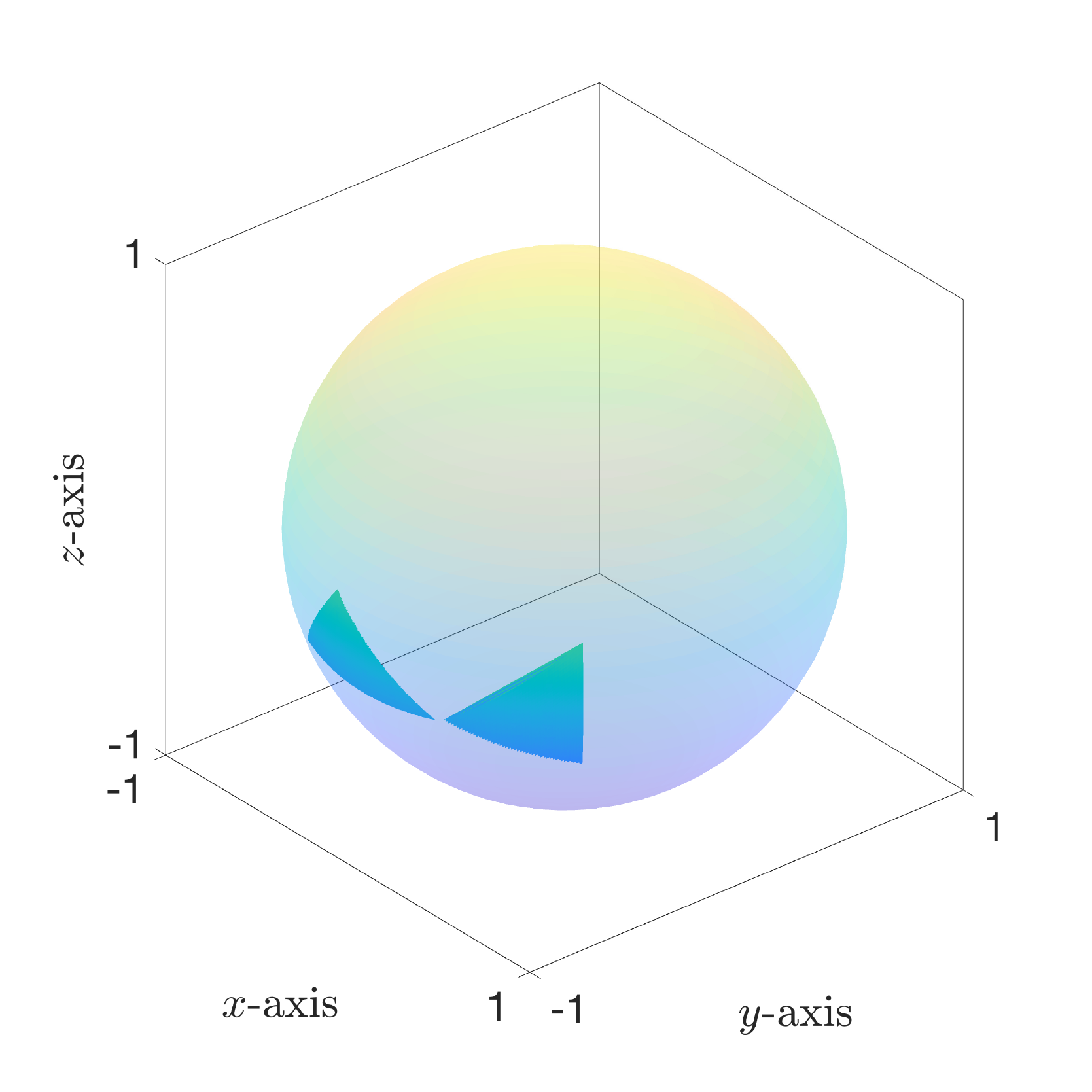}
\includegraphics[scale=\sclb]{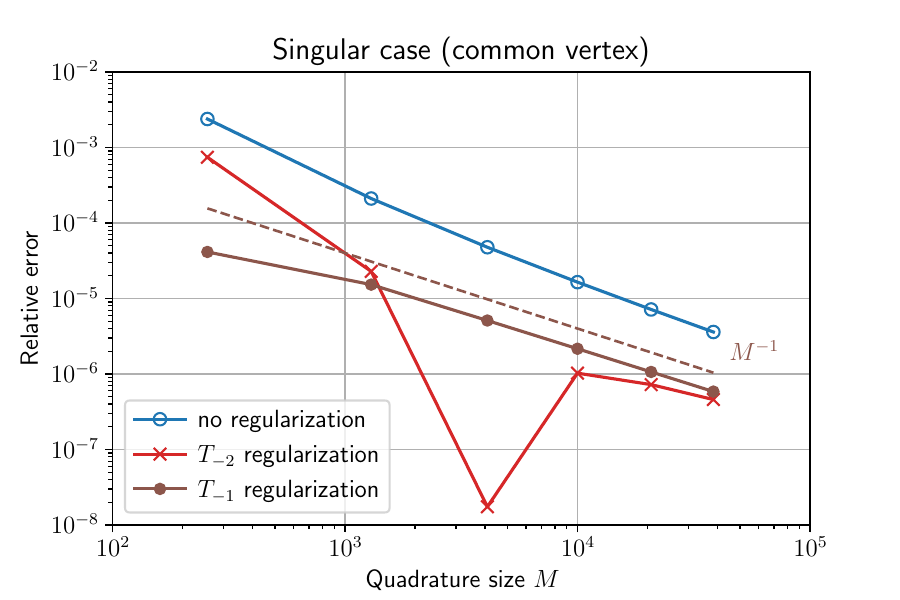}
\caption{\textit{As in the previous experiment, our method converges at a speed $\OO(M^{-0.5})$ with the total 4D number of points $M=N^2$ when using $T_{-1}$-regularization for identical triangles. For triangles sharing an edge or a vertex, the singularity is somewhat weaker and we observe superconvergence.}}
\label{fig:exp-double-layer-4D-sphere}
\end{figure}

\subsection{Boundary element computations}\label{sec:BEM}

We integrate our method into a boundary element code and solve several scattering problems in 3D to illustrate the advantage of using curved elements.

\paragraph{Scattering by a sphere} We consider the sound-soft scattering of a plane wave $u^i(r,\theta)=e^{ikr\cos\theta}$ by the unit sphere. We utilize the integral equation \cref{eq:CBIE} with $\eta=k/2$ and discretize it with a boundary element method with quadratic basis functions ($p=2)$ and quadratic triangles ($q=2$), as described in \cref{sec:bem}. We take $k=2\pi$, solve \cref{eq:CBIE} for $\varphi^s$, and evaluate the far-field pattern
\begin{align}\label{eq:far-field}
u_\infty(\bs{\theta}) = \frac{1}{4\pi}\int_{\partial D} e^{-ik\bs{\theta}\cdot\bs{y}}\varphi^s(\bs{y})dS(\bs{y}), \quad \bs{\theta}\in\mathbb{S}^2,
\end{align}
for an increasing number of triangles; see \cref{tab:dofs}. For the sphere, the exact pattern is \cite[eq.~3.32]{colton2013}
\begin{align}\label{eq:far-field-ex}
u_\infty(\theta) = \frac{i}{k}\sum_{n=0}^\infty(2n+1)\frac{j_n(k)}{h_n^{(1)}(k)}P_n(\cos\theta), \quad \theta\in[0,2\pi], 
\end{align}
with Legendre polynomials $P_n$, and spherical Bessel and Hankel functions $j_n$ and $h_n^{(1)}$. We plot the relative error in the far-field pattern in \cref{fig:cv-dofs-sphere} (left). We observe quartic superconvergence as the mesh size $h\to0$; cubic convergence was expected.\footnote{For boundary elements of degree $(p,q)$ with a mesh of size $h$, the error in the numerical far field is bounded by
\begin{align}\label{eq:far-field-error}
\vert u_\infty(\theta) - u_{\infty,h}(\theta)\vert \leq c\left(h^{2(p+1)} + h^{q+1}\right) \quad (\text{elliptic operators of order $0$}).
\end{align}
The first term corresponds to the approximation error ($p$ is the degree of the basis functions), while the second term corresponds to the geometric error ($q$ is the degree of the elements). Results of this form go back to \cite{nedelec1976}; see also~\cite{sauter2011}. For the sphere, $\OO(h^{q+1})$ seems to improve to $\OO(h^{2q})$, which is consistent with the geometric errors observed in \cite{ruiz2021}.} We also display the convergence curve for linear basis functions ($p=1$) and planar triangles ($q=1$).

\begin{table}
\caption{\textit{The triangular meshes of the unit sphere, generated by Gmsh \cite{geuzaine2009}, comprise up to a million triangles. For linear basis functions ($p=1$), the degrees of freedom (DoFs) are the function values at the vertices. For quadratic basis functions ($p=2$), the DoFs also consist in the function values at the center of the edges.}}
\centering
\ra{1.3}
\begin{tabular}{cccc}
\toprule
& & \multicolumn{2}{c}{DoFs} \\
Mesh size & Triangles & $p=1$ & $p=2$ \\
\midrule
$1.09\times10^{+0}$ & 84 & 44 & 170 \\
\num{4.48e-01} & 324 & 164 & 650 \\
\num{2.57e-01} & 1,136 & 570 & 2,274 \\
\num{1.35e-01} & 4,232 & 2,118 & 8,466 \\
\num{7.45e-02} & 16,310 & 8,157 & 32,622 \\
\num{3.43e-02} & 65,394 & 32,699 & -- \\
\num{1.80e-02} & 258,520 & 129,262 & -- \\
\num{8.99e-03} & 1,097,434 & 548,719 & -- \\
\bottomrule
\end{tabular}
\label{tab:dofs}
\end{table}

\begin{figure}
\def\scl{0.425}
\centering
\includegraphics[scale=\scl]{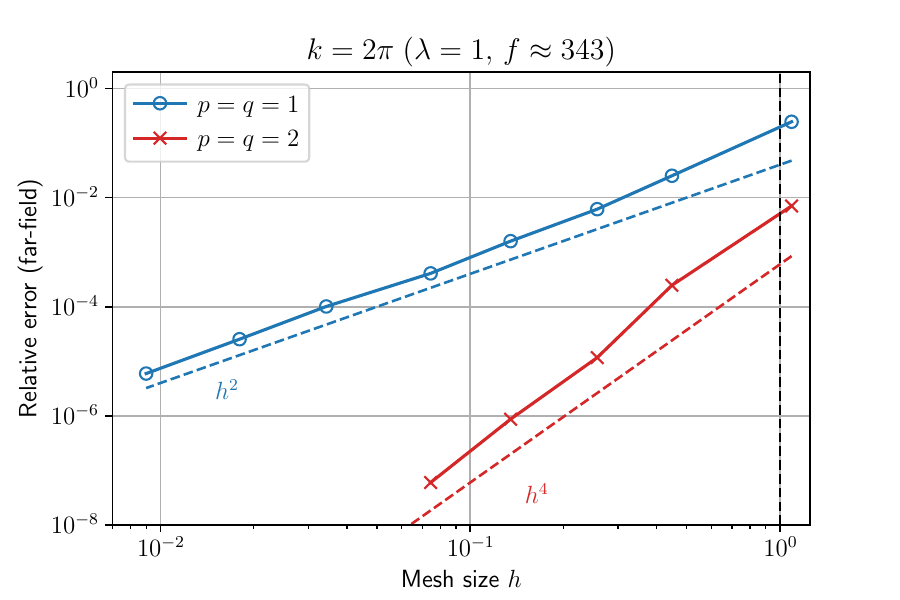}
\includegraphics[scale=\scl]{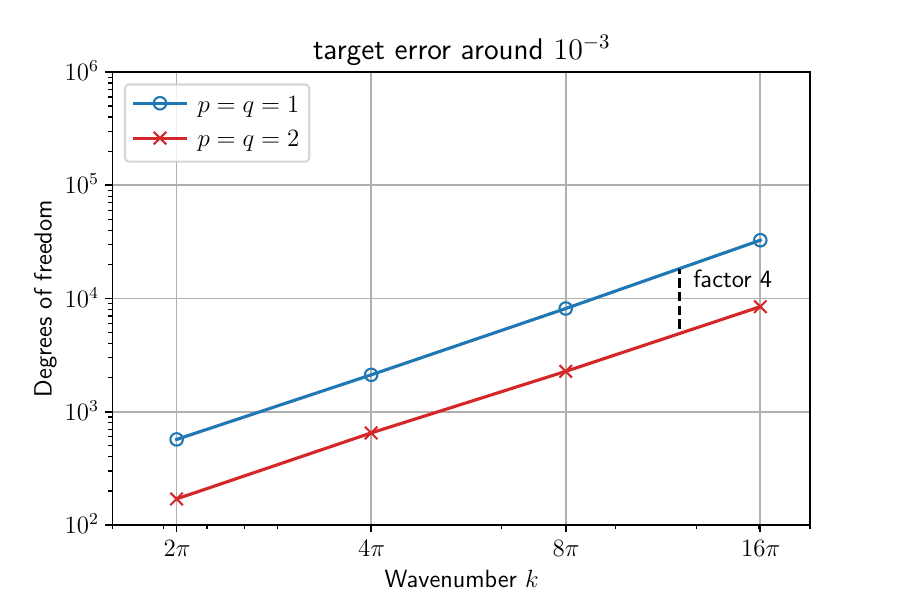}
\caption{\textit{The relative error in the far-field pattern, on the left, decreases like $h^2$ for linear basis functions and planar triangles $(p=q=1)$, while it decreases like $h^4$ for quadratic basis functions and triangles ($p=q=2$). This is a significant speed-up. A consequence of this, on the right, is that to reach a given target error of around $10^{-3}$ for a given wavenumber $k$, a method with quadratic elements needs four times fewer degrees of freedom.}}
\label{fig:cv-dofs-sphere}
\end{figure}

\begin{figure}
\def\scl{0.2}
\centering
\begin{tabular}{cc}
\includegraphics[scale=\scl]{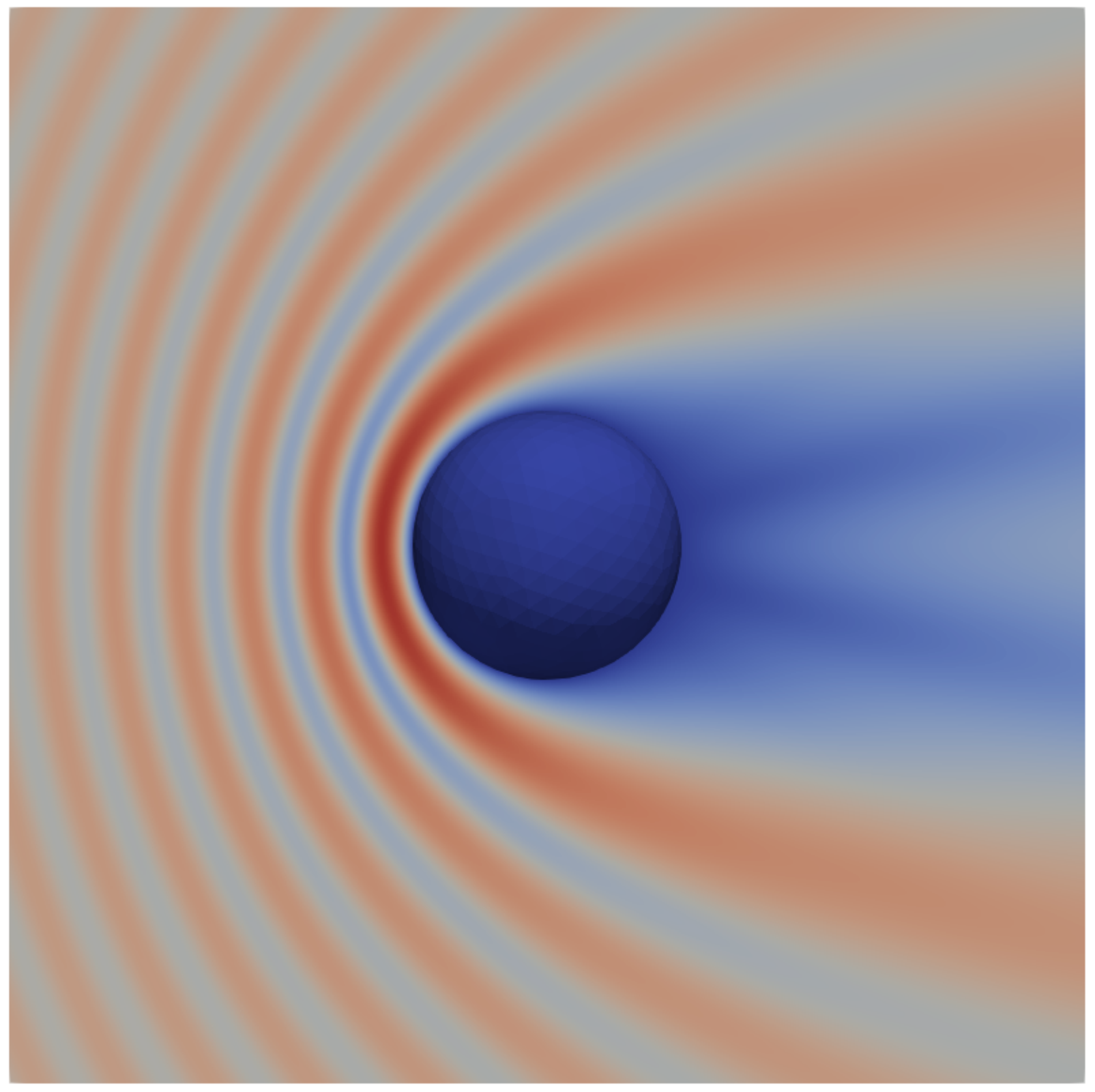} &
\includegraphics[scale=\scl]{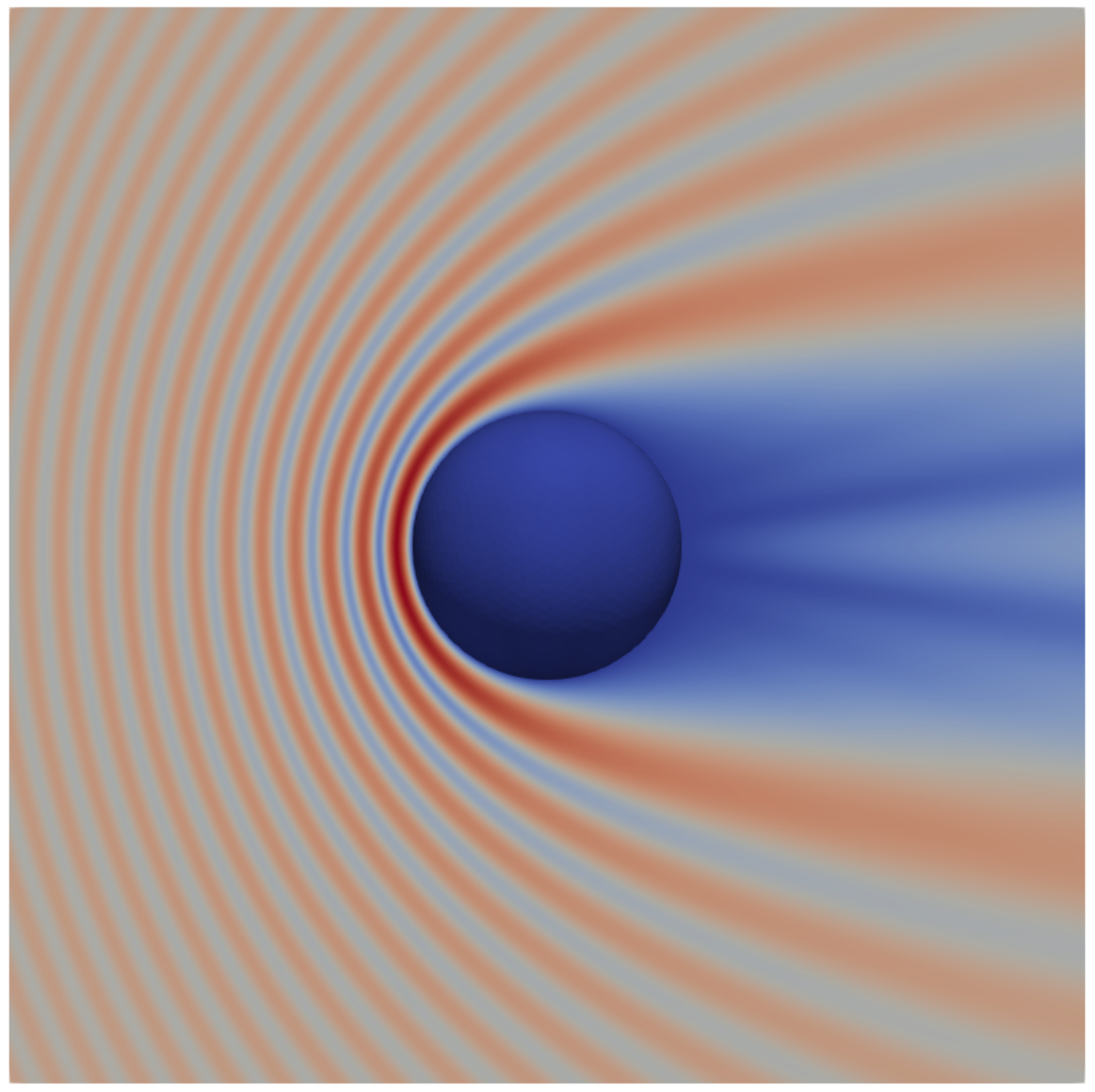} \\
\includegraphics[scale=\scl]{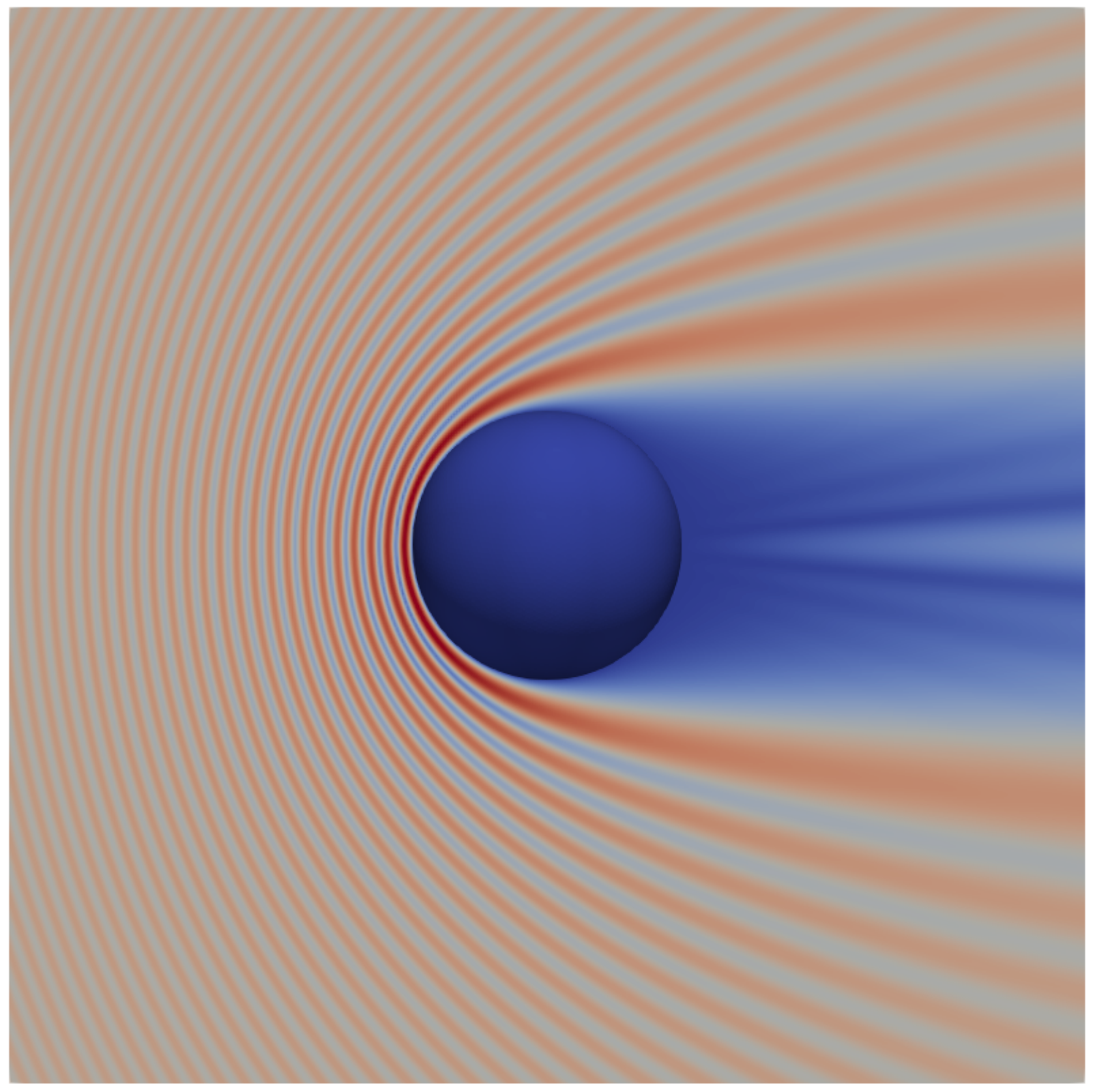} &
\includegraphics[scale=\scl]{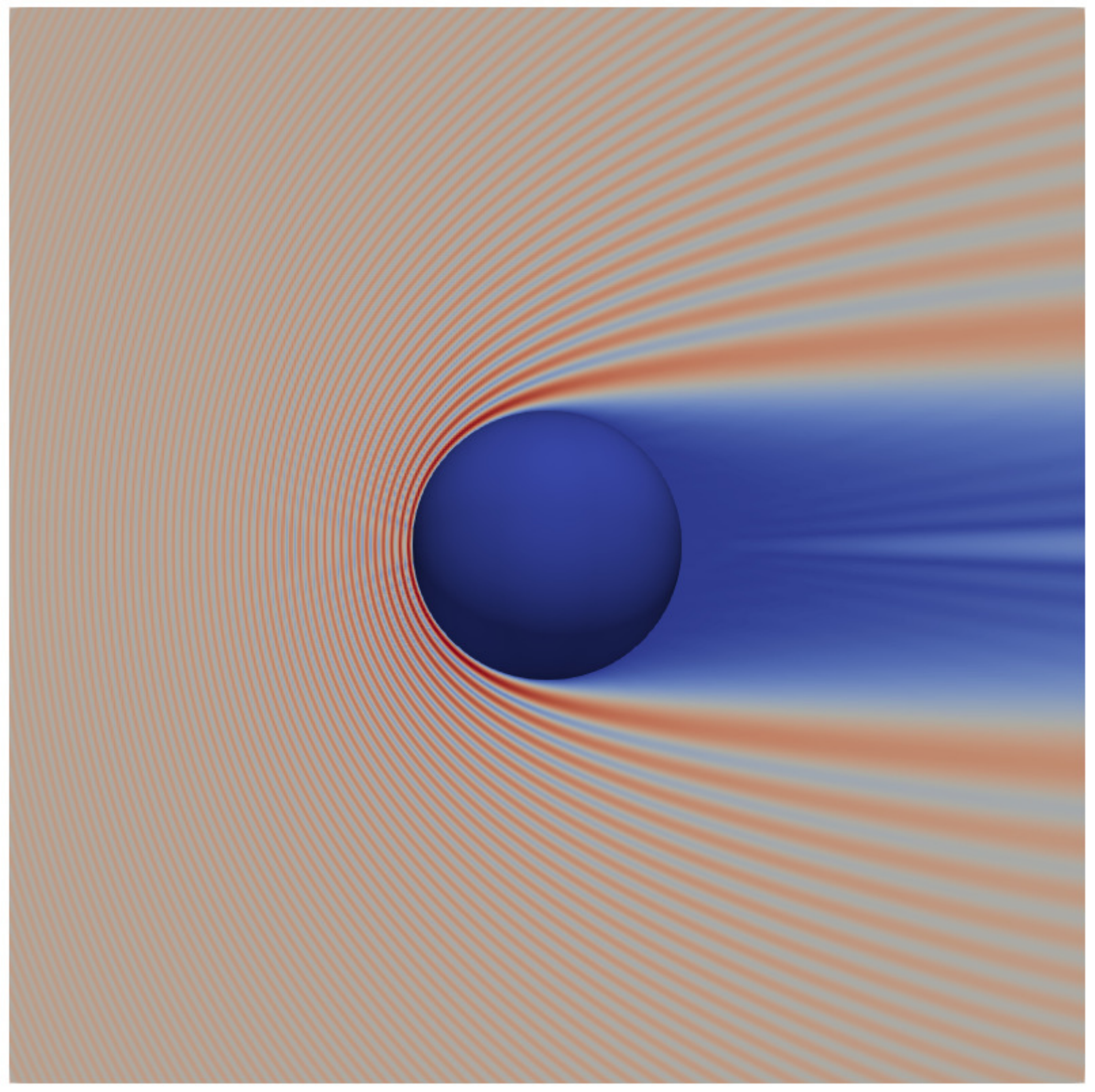}
\end{tabular}
\caption{\textit{The incident plane wave $u^i(r,\theta)=e^{ikr\cos\theta}$ is scattered by the unit sphere. We plot the amplitude of the total field $u^i+u^s$ for $k=2\pi$ (top left), $k=4\pi$ (top right), $k=8\pi$ (bottom left), and $k=16\pi$ (bottom right). The solution is computed to about six digits of accuracy with a boundary element method with quadratic triangles.}}
\label{fig:sol-sphere}
\end{figure}

We now take $k\in\{2\pi,4\pi,8\pi,16\pi\}$, solve \cref{eq:CBIE} for $\varphi^s$, and seek the number of degrees of freedom needed to reach a relative error on the far-field pattern around $10^{-3}$ for each $k$. We report the results in \cref{fig:cv-dofs-sphere} (right). We observe that using quadratic basis functions and triangles reduces the number of degrees of freedom by a factor of about four, while also decreasing computer time by a similar factor at high frequency, as seen in \cref{tab:computer-time}. All solutions are shown in \cref{fig:sol-sphere}.

\begin{table}
\caption{\textit{To solve the combined boundary integral \cref{eq:CBIE}, one has to assemble the boundary element matrices and then solve the resulting linear system---the dominant cost of the inversion is the computation of the hierarchical $LU$ factors, as discussed in the last paragraph of this section. We display below the computer time, in seconds, to obtain the target error of \cref{fig:cv-dofs-sphere} (right). We observe that using quadratic elements ($p=q=2$) decreases the total computer time by a factor up to $4.6$ at high frequency.}}
\centering
\ra{1.3}
\begin{tabular}{ccccc}
\toprule
& \multicolumn{2}{c}{$p=q=1$} & \multicolumn{2}{c}{$p=q=2$} \\
$k$ & Assembling matrices & Computing $LU$ & Assembling matrices & Computing $LU$ \\
\midrule
$2\pi$ & $1.09\mrm{e}$$-$$1$ & $6.26\mrm{e}$$-$$2$ & $7.55\mrm{e}$$+$$0$ & $4.29\mrm{e}$$-$$3$ \\
$4\pi$ & $1.48\mrm{e}$$+$$0$ & $6.54\mrm{e}$$+$$0$ & $7.39\mrm{e}$$+$$0$ & $1.02\mrm{e}$$-$$1$ \\
$8\pi$ & $8.09\mrm{e}$$+$$0$ & $7.30\mrm{e}$$+$$1$ & $2.59\mrm{e}$$+$$1$ & $1.17\mrm{e}$$+$$1$ \\
$16\pi$ & $8.55\mrm{e}$$+$$1$ & $1.96\mrm{e}$$+$$3$ & $1.40\mrm{e}$$+$$2$ & $3.02\mrm{e}$$+$$2$ \\
\bottomrule
\end{tabular}
\label{tab:computer-time}
\end{table}

\paragraph{Scattering by half-spheres} We now illustrate the robustness of our method by considering the scattering of a plane wave by two half-spheres of radius one centered at $(0,0,\pm\delta)$ for $\delta=0.5$, $\delta=0.2$, $\delta=0.1$, and $\delta=0$; see \cref{fig:sol-half-spheres-1}. Meshes were created using Gmsh, and the accuracy of the solutions was assessed by comparing them to the numerical solutions presented in \cite{montanelli2022}. The solutions are correct to about six digits of accuracy.

\begin{figure}
\def\scl{0.2}
\centering
\begin{tabular}{cc}
\includegraphics[scale=\scl]{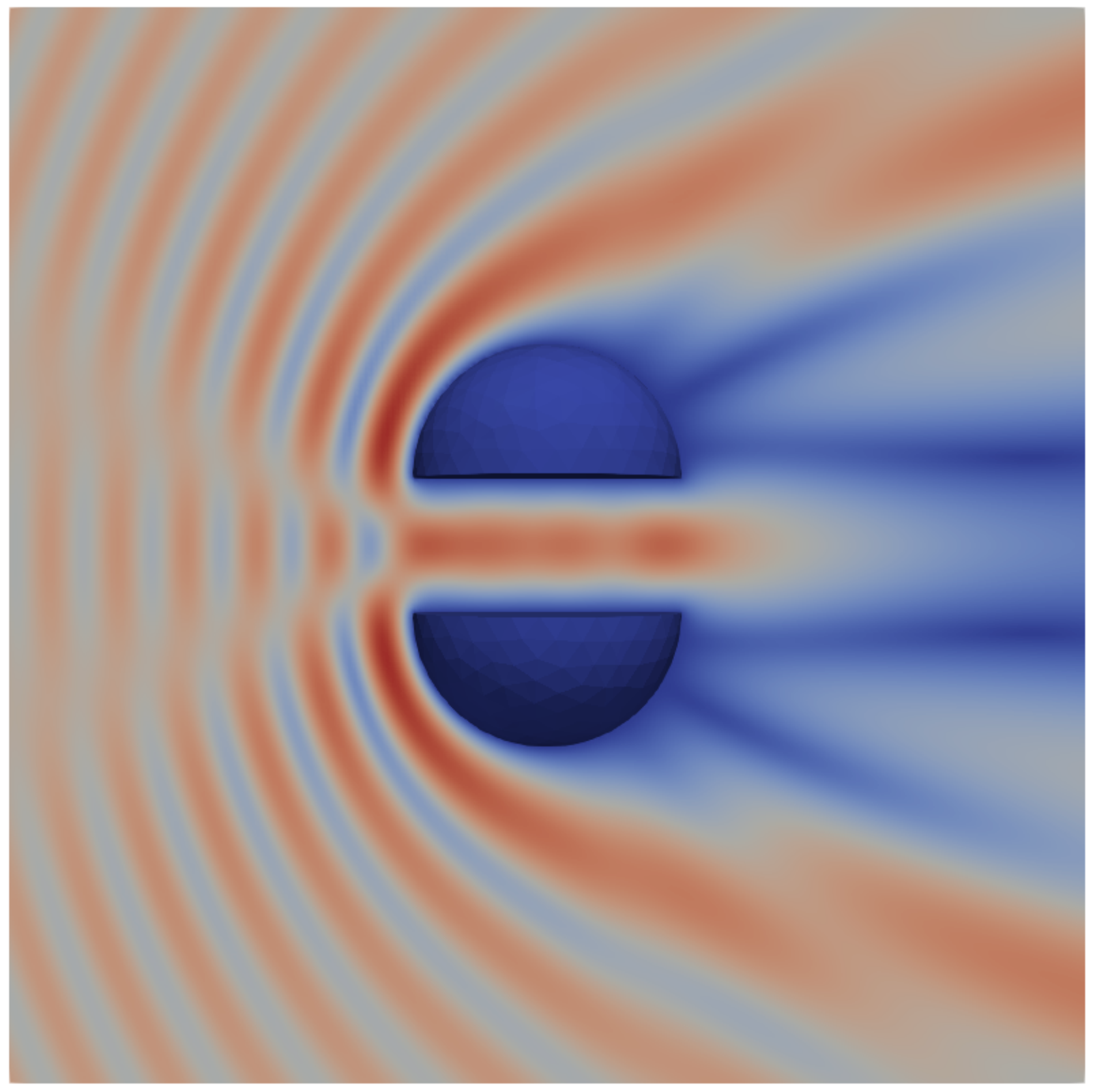} &
\includegraphics[scale=\scl]{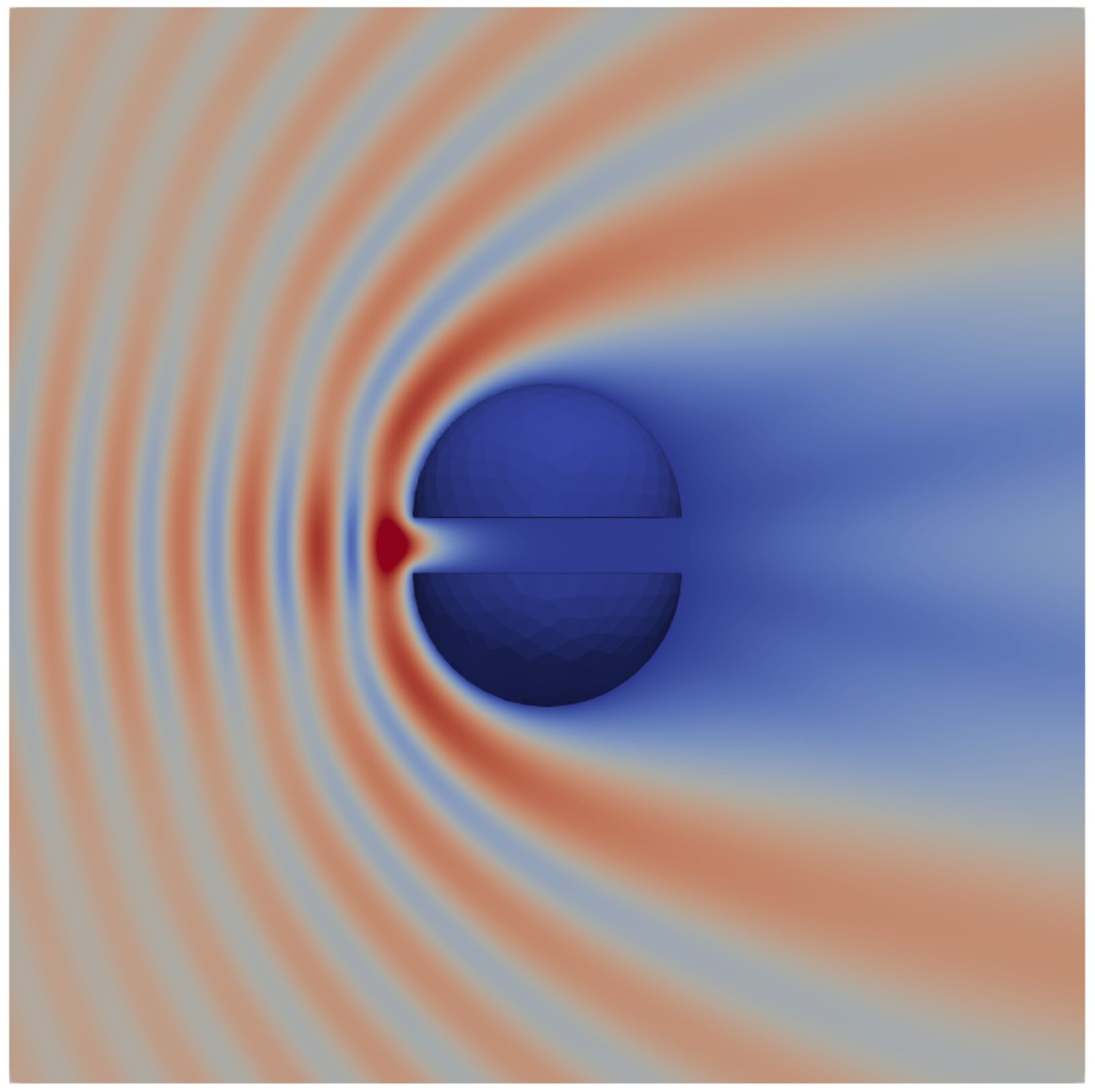} \\
\includegraphics[scale=\scl]{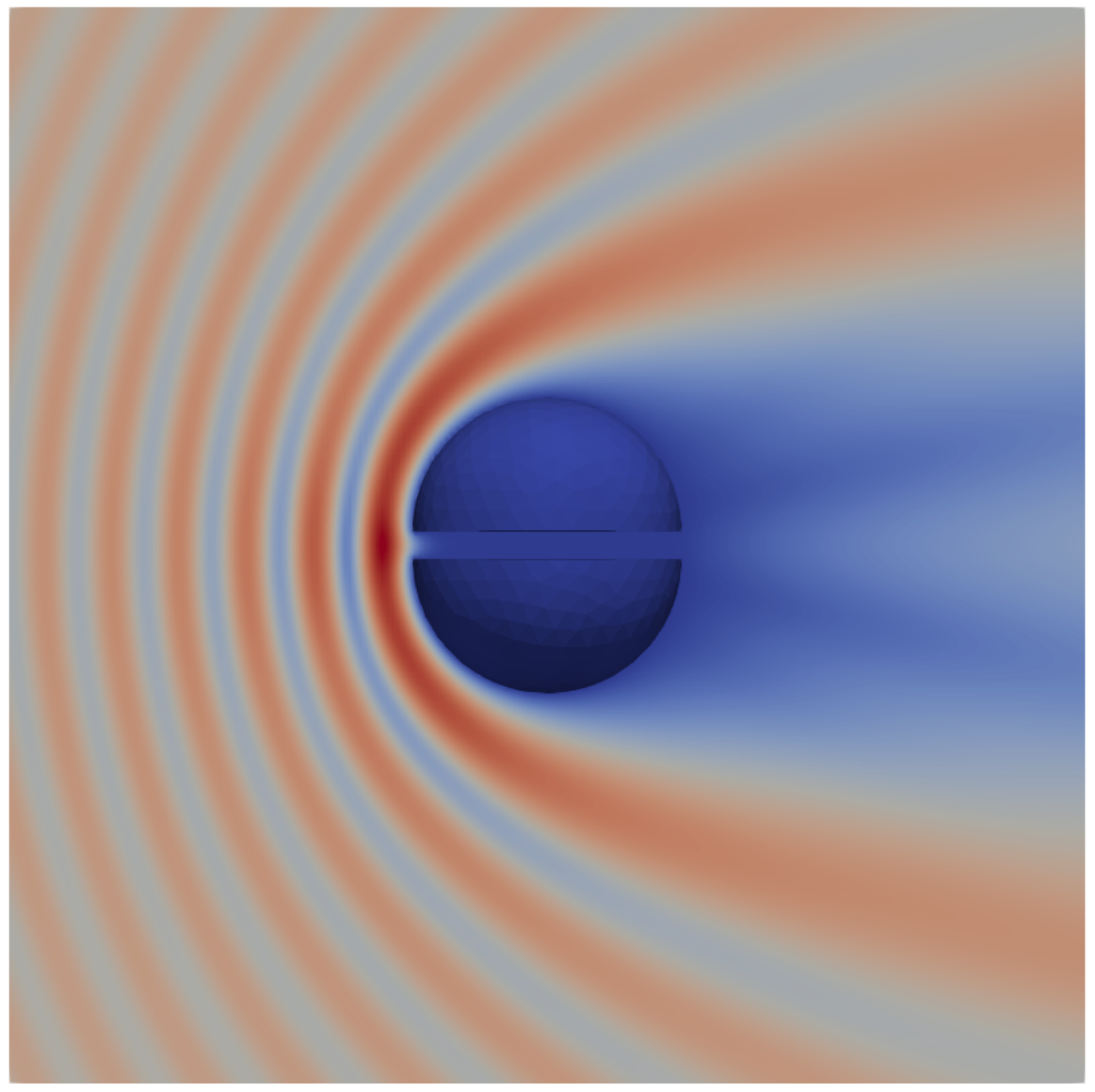} &
\includegraphics[scale=\scl]{sol-2pi.pdf}
\end{tabular}
\caption{\textit{The incident plane wave $u^i(r,\theta)=e^{ikr\cos\theta}$ is scattered by two half-spheres of radius one centered at $(0,0,\pm\delta)$. We plot the amplitude of the total field $u^i+u^s$ for $k=2\pi$ and $\delta=0.5$ (top left), $\delta=0.2$ (top right), $\delta=0.1$ (bottom left), and $\delta=0$ (bottom right). This is a challenging configuration as assembling the boundary element matrices requires to computation of many near-singular integrals.}}
\label{fig:sol-half-spheres-1}
\end{figure}

We conclude this section with a few comments on implementation. We have added our method for computing singular integrals to the C++ \href{http://leprojetcastor.gitlab.labos.polytechnique.fr/castor/}{\texttt{castor}} library of \'{E}cole Polytechnique. The \texttt{castor} library provides tools to create and manipulate matrices \textit{\`a la} MATLAB, and uses an optimized BLAS library for fast linear algebra computations. The finest mesh ($h\approx8.99\times10^{-3}$) yields dense matrices of size $10^6\times10^6$---we employed hierarchical matrices for compression, and, to solve the resulting linear systems, GMRES \cite{saad1986} preconditioned with a hierarchical $LU$ factorization at a lower precision \cite{bebendorf2005}. In this setup, the dominant cost is the computation of the $LU$ factors---GMRES typically converges to a relative residual below $10^{-10}$ in a few iterations. The computations were carried out on an Intel Xeon Gold 6154 processor (3.00 GHz, 36 cores) with 512 GB of RAM.

\section{Discussion}

We have presented in this paper a novel method for computing strongly singular and near-singular integrals based on singularity subtraction and the continuation approach. This method allows us to solve the Helmholtz exterior Dirichlet problem with a combined boundary integral equation and curved elements. We have demonstrated the accuracy and robustness of our algorithms with several experiments in 3D and have shown numerical evidence that using curved elements is advantageous---it reduces memory usage and computer time by a factor up to four.

We aim to explore various strategies to improve the efficiency of curved boundary element methods. One strategy we are considering is using the continuation approach on curved patches directly. Currently, our method involves projecting the patches onto a reference flat domain, which leads to complex formulas and the computation of 2D integrals. However, if we could use the continuation approach directly, we would only need to deal with smooth 1D integrals on the patches' boundaries, significantly reducing computational time and making curved boundary element methods even more attractive. This could be based on the generalized Stokes theorem for differential forms, similarly to \cite{zhu2022}. Lastly, one of the drawbacks of our method is the lack of rigorous justification of the convergence speeds. To address this, we plan to prove rigorous error bounds, starting with the weakly-singular case. These bounds will be based on convergence results for Gauss-Legendre quadrature rules for functions of limited regularity \cite{xiang2012}.

\appendix
\section{Elasticity potentials}\label{sec:elasticity}

Elastodynamics requires the computation of singular integrals of the form of
\begin{align}\label{eq:intsing2-elasticity}
\widetilde{I}(\bs{x}_0) = \int_{\mathcal{T}}\frac{(\bs{x}-\bs{x}_0)\cdot\bs{e}}{\vert\bs{x}-\bs{x}_0\vert^3}\varphi(F^{-1}(\bs{x}))dS(\bs{x}),
\end{align}
for some arbitrary unit vector $\bs{e}$; see, e.g., \cite{chaillat2017}. Following steps 1--2 of \cref{sec:algorithms} leads to
\begin{align}
\widetilde{I}(\bs{x}_0) = \int_{\widehat{T}}\frac{(F(\bs{\hat{x}})-F(\bs{\hat{x}}_0)-h\bs{\hat{n}}_0/\vert\bs{\hat{n}}_0\vert)\cdot\vert\bs{\hat{n}}(\bs{\hat{x}})\vert\bs{e}}{\vert F(\bs{\hat{x}})-\bs{x}_0\vert^3}\varphi(\bs{\hat{x}})dS(\bs{\hat{x}}).
\end{align}
We compute the asymptotic term $\widetilde{T}_{-2}$,
\begin{align}\label{eq:Tn2-elasticity}
\widetilde{T}_{-2}(\bs{\hat{x}},h) = \frac{\left[J_0(\bs{\hat{x}}-\bs{\hat{x}}_0) - h\bs{\hat{n}}_0/\vert\bs{\hat{n}}_0\vert\right]\cdot\vert\bs{\hat{n}}_0\vert\bs{e}}{\left[\vert J_0(\bs{\hat{x}} - \bs{\hat{x}}_0)\vert^2 + h^2\right]^{\frac{3}{2}}}\varphi_0.
\end{align}
For the new term in \cref{eq:Tn2-elasticity}, we apply \cref{thm:continuation} with $\ell=0$. This yields
\begin{align}
\widetilde{I}_{-2}(\bs{x}_0) = \, & \varphi_0\sum_{j=1}^3\hat{s}_j\int_{\partial\widehat{T}_j-\bs{\hat{x}}_0}\left(J_0\bs{\hat{x}}\cdot\vert\bs{\hat{n}}_0\vert\bs{e}\right)\widetilde{g}_{-2}(\bs{\hat{x}},h)ds(\bs{\hat{x}})\\
& -\mrm{sign}(h)\varphi_0\left(\bs{\hat{n}}_0\cdot\bs{e}\right)\sum_{j=1}^3\hat{s}_j\int_{\partial\widehat{T}_j-\bs{\hat{x}}_0}\frac{\sqrt{\nrm^2 + h^2} - \vert h\vert}{\nrm^2\sqrt{\nrm^2 + h^2}}ds(\bs{\hat{x}}), \nonumber
\end{align}
where the one-dimensional integrand $\widetilde{g}_{-2}$ is defined by
\begin{align}\label{eq:gn2-elasticity}
\widetilde{g}_{-2}(\bs{\hat{x}},h) = -\frac{1}{\vert J_0\bs{\hat{x}}\vert^2\sqrt{\vert J_0\bs{\hat{x}}\vert^2 + h^2}} + \frac{\mrm{arcsinh}\left(\frac{\vert J_0\bs{\hat{x}}\vert}{\vert h\vert}\right)}{\vert J_0\bs{\hat{x}}\vert^3}.
\end{align}
When $h=0$, the first term in \cref{eq:gn2-elasticity} integrates to $0$ (it is the residue of \cref{thm:continuation}), while the second term may be replaced by $\log\vert\bs{\hat{x}}\vert/\vert J_0\bs{\hat{x}}\vert^3$ (it is the Cauchy principal value).

\section*{Acknowledgments}

We express our gratitude to Laurent Series for his valuable assistance in performing the computations at the CMAP department of \'{E}cole Polytechnique. We also thank the members of the Inria GAMMA research team, in particular Lucien Rochery and Matthieu Manoury, for providing us several boundary element meshes with quadratic triangles.

\bibliographystyle{siamplain}
\bibliography{/Users/montanelli/Dropbox/HM/WORK/ACADEMIA/BIBLIOGRAPHY/references.bib}

\end{document}